\documentclass{article}
\usepackage{epsfig}
\usepackage{amsmath,amsthm}
\usepackage{amssymb}
\usepackage{graphicx,subfigure,epstopdf,latexsym,amsfonts,fancyhdr}
\usepackage{booktabs,arydshln,multirow}
\usepackage{setspace}
\usepackage{xcolor}
\usepackage{stmaryrd}
\usepackage{cite}
\usepackage[utf8]{inputenc}
\newtheorem{thm}{Theorem}[section]

\newtheorem{lemma}[thm]{Lemma}

\newtheorem{rem}[thm]{Remark}

\numberwithin{equation}{section} \topmargin=-2.5cm \oddsidemargin=0.4cm
\newcommand{\normmm}[1]{{\left\vert\kern-0.25ex\left\vert\kern-0.25ex\left\vert #1
    \right\vert\kern-0.25ex\right\vert\kern-0.25ex\right\vert}}
\begin{document}
\title{Efficient shifted fractional trapezoidal rule for subdiffusion problems with nonsmooth solutions on uniform meshes
\thanks{Corresponding author.
\newline \emph{Email addresses:} mathliuyang@imu.edu.cn,
\newline \emph{Manuscript~~~~~~~~~~~~~~~~ Sept. 2020}
}}
\date{ }
\author{Baoli Yin$^1$, Yang Liu$^{1*}$, Hong Li$^1$, Zhimin Zhang$^{2,3}$
\\\small{\emph{$^1$School of Mathematical Sciences, Inner Mongolia University, Hohhot 010021,
China;}}
\\\small{\emph{$^2$Beijing Computational Science Research Center, Beijing 100193, China;}}
\\\small{\emph{$^3$Department of Mathematics, Wayne State University, Detroit, MI 48202, USA}}
}
\date{}
 \maketitle
  {\color{black}\noindent\rule[0.5\baselineskip]{\textwidth}{0.5pt} }
\noindent \textbf{Abstract:}
This article devotes to developing robust but simple correction techniques and efficient algorithms for a class of second-order time stepping methods, namely the shifted fractional trapezoidal rule (SFTR), for subdiffusion problems to resolve the initial singularity and nonlocality.
The stability analysis and sharp error estimates in terms of the smoothness of the initial data and source term are presented.
As a byproduct in numerical tests, we find amazingly that the Crank-Nicolson scheme ($\theta=\frac{1}{2}$) without initial corrections can restore the optimal convergence rate for the subdiffusion problem with smooth initial data and source terms.
To deal with the nonlocality, fast algorithms are considered to reduce the computational cost from $O(N^2)$ to $O(N \log N)$ and save the memory storage from $O(N)$ to $O(\log N)$, where $N$ denotes the number of time levels.
Numerical tests are performed to verify the sharpness of the theoretical results and confirm the efficiency and accuracy of initial corrections and the fast algorithms.
\\
 {\color{black}\noindent\rule[0.5\baselineskip]{\textwidth}{0.5pt} }
\def\REF#1{\par\hangindent\parindent\indent\llap{#1\enspace}\ignorespaces}
\newcommand{\h}{\hspace{1.cm}}
\newcommand{\hh}{\hspace{2.cm}}
\newtheorem{yl}{\hspace{0.cm}Lemma}
\newtheorem{dl}{\hspace{0.cm}Theorem}
\newtheorem{re}{\hspace{0.cm}Remark}
\renewcommand{\sec}{\section*}
\renewcommand{\l}{\langle}
\renewcommand{\r}{\rangle}
\newcommand{\be}{\begin{eqnarray}}
\newcommand{\ee}{\end{eqnarray}}
\normalsize \vskip 0.2in
\section{Introduction}\label{sec.int}
In this paper, we consider the subdiffusion problem
\begin{equation}\label{I.1}\begin{split}
\begin{cases}
  \partial_t^\alpha u(\boldsymbol x,t)-\Delta u(\boldsymbol x,t)=f(\boldsymbol x,t), & (\boldsymbol x,t)\in \Omega\times (0,T], \\
  u(\boldsymbol x,t)=0, & \boldsymbol x \in \partial\Omega, ~t\in (0,T],\\
  u(\boldsymbol x,0)=v(\boldsymbol x), & \boldsymbol x \in \Omega,
\end{cases}
\end{split}\end{equation}
where $T>0$, $\Omega$ is a bounded convex polygonal domain in $\mathbb{R}^d$ ($d=1,2,3$) with a boundary $\partial \Omega$.
$\Delta$ denotes the Laplacian that maps $H_0^1(\Omega)\cap H^2(\Omega)$ to $L^2(\Omega)$, and $f:(0,T]\to L^2(\Omega)$ is a given function.
The initial data $v\in D(\Delta):=H_0^1(\Omega)\cap H^2(\Omega)$ or $v\in L^2(\Omega)$.
$\partial_t^\alpha$ ($\alpha \in (0,1)$) is the Caputo fractional operator satisfying $\partial_t^\alpha \psi=D_t^\alpha (\psi-\psi(0))$ where $D_t^\alpha$ denotes the Riemann-Liouville operator defined by
\begin{equation}\label{I.2}\begin{split}
D_t^\alpha \psi(t)=\frac{1}{\Gamma(1-\alpha)}\frac{\mathrm{d}}{\mathrm{d}t}
\int_{0}^{t}(t-s)^{-\alpha}\psi(s)\mathrm{d}s,
\quad
\alpha \in (0,1).
\end{split}\end{equation}
\par
The subdiffusion problem (\ref{I.1}) in recent years has been studied extensively since as the simplest model of evolution equations that involve time fractional derivatives, it was successfully applied in engineering, physics, biology and finance, and show better description of physical phenomenons compared with the traditional model.
The method adopted in this paper to discretize temporal direction is known as shifted fractional trapezoid rule (SFTR) first proposed in \cite{LiuYin3}, which belongs to time-stepping methods.
Historically, several groups of time-stepping methods have been devised such as the convolution quadrature (CQ)\cite{Lubich1}, the $L$ type formulas ($L1$, $L1$-$2$, $L2$-$1_\sigma$)\cite{Diethelm,OldhamSpanier,SunWu,LinXuL1,GaoSunZhang,Alikhanov}, the weighted and shifted Gr\"{u}nwald difference operators (WSGD)\cite{TianZhouDeng,LiCai} and shifted convolution quadrature (SCQ)\cite{LiuYin1,DingLiYi}, and so on.
Generally, since the subdiffusion problem (\ref{I.1}) is characterized by the solution singularity at initial time\cite{Stynes1,SakamotoYamamoto}, researchers have paid special attention on developing novel techniques such as by using nonuniform meshes \cite{Stynes2,LiaoMcLeanZhang,KMustapha} or adding correction terms \cite{JinLiZh1,Yan1,JinLiZhou4,Zeng2,Yan2} to restore the optimal convergence rate.
\par
However, the study of time-stepping methods for problem (\ref{I.1}) is still inadequate, particularly for SCQ methods.
The authors in \cite{JinLiZh1} considered the fractional Crank-Nicolson scheme for (\ref{I.1}) which can be viewed as a special SCQ method approximating the fractional derivative at shifted node $t_{n-\theta}$ with $\theta=\frac{\alpha}{2}$.
A two-step correction technique was provided and sharp error estimates were proposed.
Inspired by \cite{JinLiZh1}, we shall adopt the SFTR which possesses a free shifted parameter $\theta$ and reduces to the fractional Crank-Nicolson scheme when $\theta=\frac{\alpha}{2}$ as studied in \cite{JinLiZh1}.
Moreover, we devise a single-step correction technique with correction coefficients determined by $\theta$, and thus is simpler than that considered in \cite{JinLiZh1} when $\theta=\frac{\alpha}{2}$.
We remark that numerical analysis for the scheme obtained by SFTR is nontrivial due to the complexity of introducing a free parameter $\theta$ and further, the independence of $\theta$ on $\alpha$ allows us to develop $\alpha$-robust theoretical results, i.e., the error bound remains finite when $\alpha\to 1^-$\cite{ChenStynes}.
As a byproduct, we find amazingly that the correction coefficients vanish if $\theta=\frac{1}{2}$, which corresponds to the Crank-Nicolson scheme, meaning that \textit{the initial singularity can be resolved by approximations at a shifted node}.
Although this phenomenon is verified only numerically in this article since the numerical analysis excludes the special case $\theta=\frac{1}{2}$ essentially, this is the first appearance of such a proposition for subdiffusion problem in literature to the best of authors' knowledge.
A similar conclusion was derived in \cite{GunzburgerWang} for fractional PDEs which were dominated by a first-order derivative.
\par
The nonlocality of the fractional derivative makes the computation rather time consuming and memory expensive.
Scholars have proposed different fast algorithms to save computing costs, see for example \cite{BaffetHesthaven,Schadle,Zeng1,JiangZhangZhang} and references therein.
Nevertheless, fast algorithms concerning SCQ methods are limited especially for the study on the choice of the shifted parameter $\theta$.
In this study, we consider two different fast algorithms for the fully discrete scheme.
Both algorithms can reduce the computing complexity from $O(N^2)$ to $O(N\log N)$ and memory requirement from $O(N)$ to $O(\log N)$ with fixed space mesh.
We further discuss the choice of $\theta$ to make sure the error incurred by fast algorithms is as small as possible.
\par
The contributions of this article can be summarized as follows.
\par
(i) Propose simple correction techniques for SFTR and reveal that correction coefficients are determined only by the shifted parameter $\theta$.
\par
(ii) Develop sharp and $\alpha$-robust error estimates for the proposed scheme depending on the smoothness of the initial data and source term.
\par
(iii) Fast algorithms are considered to reduce the computing cost including CPU time and memory requirement.
The experimental results concerning the choice of $\theta$ are meaningful for practical applications.
\par
(iv) Several numerical experiments are implemented to verify the correctness and sharpness of our theoretical results.
The Crank-Nicolson scheme combined with SFTR for (\ref{I.1}) is discovered to preserve the optimal convergence order if the initial data and source term are smooth enough.
\par
We organize the article in several sections.
In Sect.\ref{sec.sec}, we introduce the SFTR and formulate the fully discrete scheme with a novel correction technique.
Sect.\ref{sec.sta} devotes to the stability analysis.
In Sect.\ref{sec.sha}, sharp error estimates are provided for our scheme.
To save computing expenses, two different fast algorithms are considered in Sect.\ref{sec.fast} where some discussion on the choice of $\theta$ is given.
In Sect.\ref{sec.num}, several tests are conducted to verify the sharpness of our theoretical results and the efficiency of fast algorithms.
\par
Throughout, the symbol $c$, with or without a subscript, is regarded as a generic positive constant which may change at different occurrences.
However, we require $c$ is always independent of mesh sizes $h$ and $\tau$.
\section{Second-order schemes}\label{sec.sec}
As our interest mainly focuses on the analysis of time discretization using time-stepping methods, we first formulate the space semidiscrete scheme by the finite element method.
Denote by $\mathfrak{T}_h$ a shape regular, quasi-uniform triangulation of the domain $\Omega$ into $d$-simplexes with a mesh size $h$.
In each simplex $e$, we approximate the solution by a linear polynomial function.
Introduce the space
\begin{equation}\label{Se.1}\begin{split}
V_h=\{\chi_h \in H_0^1(\Omega): \chi_h |_e \text{ is a linear polynomial function},~ e \in \mathfrak{T}_h\}.
\end{split}\end{equation}
Define the $L^2(\Omega)$ projection $P_h:L^2(\Omega)\to V_h$, the Ritz projection $R_h:H_0^1(\Omega)\to V_h$ and the discrete Laplacian $\Delta_h: V_h\to V_h$ as follows,
\begin{equation}\label{Se.2}\begin{split}
(P_h \phi,\chi_h)&=(\phi,\chi_h), \quad\forall \phi \in L^2(\Omega),~ \forall \chi_h \in V_h,
\\
(\nabla R_h \phi,\nabla \chi_h)&=(\nabla \phi,\nabla \chi_h),\quad \forall \phi \in H_0^1(\Omega),~ \forall \chi_h \in V_h,
\\
(\Delta_h \phi_h,\chi_h)&=-(\nabla \phi_h, \nabla \chi_h), \quad \forall \phi_h, \chi_h \in V_h,
\end{split}\end{equation}
where $(\cdot,\cdot)$ is the inner product of $L^2(\Omega)$ space.
Thus, the semidiscrete scheme in space for problem (\ref{I.1}) is to find $u_h \in V_h$ such that for any $\chi_h \in V_h$, there holds
\begin{equation}\label{Se.3}\begin{split}
(\partial_t^\alpha u_h,\chi_h)+(\nabla u_h,\nabla \chi_h)=(f,\chi_h),
\end{split}\end{equation}
with the initial condition \cite{JinLiZh1,Thomee1}
\begin{equation}\label{Se.4}\begin{split}
u_h(0)=v_h:=
\begin{cases}
  P_h v, & \mbox{if } v \in L^2(\Omega), \\
  R_h v, & \mbox{if } v \in D(\Delta).
\end{cases}
\end{split}\end{equation}
By $\Delta_h$ and letting $w_h:=u_h-v_h$, we rewrite (\ref{Se.3}) as
\begin{equation}\label{Se.5}\begin{split}
D_t^\alpha w_h(t)-\Delta_h w_h(t)=f_h(t)+\Delta_h v_h,\quad t>0,
\end{split}\end{equation}
where $f_h:=P_hf$.
\par
To formulate the fully discrete scheme, divide the time interval by grids: $0=t_0<t_1<\cdots<t_N=T$ with $t_k=k\tau$, $\tau=T/N$.
Let $w_h^n=w_h(t_n)$ and assume $W_h^n$ is the approximation to $w_h^n$.
Denote by $\psi^{n-\theta}, D_{\tau}^\alpha \psi^{n-\theta}$ the approximation of $\psi(t_{n-\theta}), D_t^\alpha \psi(t_{n-\theta})$, respectively, as the following
\begin{equation}\label{Se.6}\begin{split}
\psi^{n-\theta}&:=(1-\theta)\psi^n+\theta\psi^{n-1},
\\
D_{\tau}^\alpha \psi^{n-\theta}&:=\tau^{-\alpha}\sum_{k=0}^{n}\omega_k \psi^{n-k},
\end{split}\end{equation}
where the weights $\omega_k$, generally depending on $\alpha$ and $\theta$, are coefficients of a generating function in the form $\omega(\xi)=\sum_{k=0}^{\infty}\omega_k\xi^k$.
Using Taylor expansion, we can obtain $\psi(t_{n-\theta})=\psi^{n-\theta}+O(\tau^2)$ if $\psi$ is smooth.

\begin{lemma}\label{lem.1}(Shifted fractional trapezoidal rule, see \cite{LiuYin3})
Assume $\psi(t)$ is a smooth function satisfying $\psi(0)=0$.
Then, we have
\begin{equation}\label{Se.7}\begin{split}
D_t^\alpha \psi(t_{n-\theta})=D_\tau^\alpha\psi^{n-\theta}+O(\tau^2),
\end{split}\end{equation}
where the weights $(\omega_k)_{k=0}^N$ are defined by the generating function
\begin{equation}\label{Se.8}\begin{split}
\omega(\xi)=\bigg[\frac{1-\xi}{\frac{1}{2}(1+\xi)+\frac{\theta}{\alpha}(1-\xi)}\bigg]^\alpha,
  \quad 0< \theta \leq \frac{1}{2}.
\end{split}\end{equation}
\end{lemma}
\begin{rem}\label{rem.2}
We take a note that for $\theta >\frac{1}{2}$, resultant schemes may be unstable, see \textit{Example 3} in Sect.\ref{sec.num}.
If $\theta=0$, then (\ref{Se.8}) actually recovers the generating function of the fractional trapezoidal rule developed in the convolution quadrature (CQ) \cite{Lubich1}.
To efficiently obtain the weights $(\omega_k)_{k=0}^N$, we employ the following recursive formulas \cite{LiuYin3},
\begin{equation}\label{Se.8.1}
\omega_k=
\begin{cases}
  \big(\frac{2\alpha}{\alpha+2\theta}\big)^{\alpha}, & \mbox{if } k=0, \\
  -\alpha\big(\frac{2\alpha}{\alpha+2\theta}\big)^{\alpha+1}, & \mbox{if } k=1, \\
  \frac{2\alpha}{k(\alpha+2\theta)}\big\{\big[\frac{2\theta}{\alpha}(k-1)-\alpha\big]\omega_{k-1}+\frac{\alpha-2\theta}{2\alpha}(k-2)\omega_{k-2}\big\}, & \mbox{if } k \geq 2.
\end{cases}
\end{equation}
\end{rem}
\par
\textbf{Fully discrete scheme without corrections}
\par
By (\ref{Se.6}) and Lemma \ref{lem.1}, we get the fully discrete scheme: Find $W_h^n \in V_h$ such that,
\begin{equation}\label{Se.9}\begin{split}
D_\tau^\alpha W_h^{n-\theta}-\Delta_h W_h^{n-\theta}=f_h^{n-\theta}+\Delta_h v_h,\quad n\geq 1.
\end{split}\end{equation}
However, due to the initial singularity \cite{SakamotoYamamoto,Stynes1} of solutions of (\ref{I.1}), we can only obtain first-order convergence rate at fixed time when $\theta<\frac{1}{2}$, instead of the desired second-order one, see Table \ref{tab1}.
Inspired by the works \cite{JinLiZh1,CuestaLubich}, we shall in the following analysis add initial corrections to the first time step to resolve this problem.
\par
\textbf{Fully discrete scheme with initial corrections}
\par
By adding initial corrections, the fully discrete scheme can be stated as: Find $W_h^n \in V_h$ such that
\begin{equation}\label{Se.14}\begin{split}
D_\tau^\alpha W_h^{1-\theta}-\Delta_h W_h^{1-\theta}&=(3/2-\theta)\Delta_h v_h+\frac{1}{2}f_h^0+(1-\theta)f_h^1,
\\
D_\tau^\alpha W_h^{n-\theta}-\Delta_h W_h^{n-\theta}&=f_h^{n-\theta}+\Delta_h v_h,\quad n\geq 2.
\end{split}\end{equation}
Define $U_h^n=W_h^n+v_h$, we can rewrite (\ref{Se.14}) as
\begin{equation}\label{Se.14.1}\begin{split}
D_\tau^\alpha (U_h-v_h)^{1-\theta}-\Delta_h U_h^{1-\theta}&=(1/2-\theta)(f_h^0+\Delta_h v_h)+f_h^{1-\theta},
\\
D_\tau^\alpha (U_h-v_h)^{n-\theta}-\Delta_h U_h^{n-\theta}&=f_h^{n-\theta},\quad n\geq 2.
\end{split}\end{equation}
\begin{rem}\label{rem.4}
Note that (\ref{Se.9}) reduces to the fractional Crank-Nicolson scheme when $\theta=\frac{\alpha}{2}$ which was studied in \cite{JinLiZh1} where a two-step correction technique is adopted to deal with initial singularities.
Our scheme (\ref{Se.14}) requires only to fix the first step, and thus is easier to implement.
Numerical tests show that both correction techniques are efficient, see \textit{Example 5} in Sect.\ref{sec.num}.
Moreover, by taking a limiting process $\theta\to \frac{1}{2}$, (\ref{Se.14}) becomes exactly the following Crank-Nicolson scheme, without any initial correction.
\end{rem}
\par
\textbf{Crank-Nicolson scheme}
\par
Let $\theta\to\frac{1}{2}$ in (\ref{Se.14}), we obtain
\begin{equation}\label{Se.15}\begin{split}
D_\tau^\alpha W_h^{n-1/2}-\Delta_h W_h^{n-1/2}=f_h^{n-1/2}+\Delta_h v_h,\quad n\geq 1.
\end{split}\end{equation}
One can find the scheme (\ref{Se.15}) is also a special case of (\ref{Se.9}), i.e., the scheme without any initial correction.
Numerical results in Table \ref{tab1} with $\theta=\frac{1}{2}$ show that (\ref{Se.15}) can maintain the optimal convergence rate at fixed time for smooth initial data and source terms.
However, our theoretical analysis for the scheme (\ref{Se.14}) in later sections  exclude the special case $\theta=\frac{1}{2}$, since the theory itself demands some analyticity of functions $\kappa(\xi)$ and $\beta_\tau(\xi)$ (defined in Theorem \ref{thm.6}) at $\xi=-1$ in the complex plane.
See also the Remark \ref{rem.8}.
\section{Stability analysis}\label{sec.sta}
In this section, we give the stability analysis based on the maximal $\ell^p$-regularity and discrete Gr\"{o}nwall type inequality for scheme (\ref{Se.14}).
\par
Some definitions and lemmas are needed for later analysis.
Throughout, we denote by $X$ the space $L^q(\Omega)$ or subspace $V_h$ endowed with the $L^q$ norm, $1<q<\infty$, by $\mathcal{B}(X)$ the set of all bounded linear operators from $X$ to itself.
For given $\sigma \in (0,\pi)$, define the open sector $\Sigma_\sigma:=\{z\in \mathbb{C}:|\arg z|<\sigma, z\neq 0\}$, and the unit circle $\mathbb{D}:=\{z\in \mathbb{C}:|z|=1\}$.
Let $\mathbb{D}':=\mathbb{D}\setminus \{\pm 1\}$.
Define the space $\ell^p(X)$ of sequence $(f^n)_{n=0}^\infty$ with $f^n \in X, n=0,1,\cdots$, such that $\|(f^n)_{n=0}^\infty\|_{\ell(X)}<\infty$, where
\begin{equation}\label{St.1}\begin{split}
\|(f^n)_{n=0}^\infty\|_{\ell^p(X)}:=
\begin{cases}
  \bigg(\tau\displaystyle\sum_{n=0}^{\infty}\|f^n\|_X^p\bigg)^\frac{1}{p}, & \mbox{if } 1\leq p <\infty \\
  \displaystyle\sup_{n\geq 0}\|f^n\|_X, & \mbox{if }p=\infty.
\end{cases}
\end{split}\end{equation}
For a sequence $(f^n)_{n=0}^N$, define $\|(f^n)_{n=0}^N\|_{\ell^p(X)}=\|(f^n)_{n=0}^\infty\|_{\ell^p(X)}$ by setting $f^n=0$ for $n>N$.
\begin{def}\label{def.1}(R-boundedness, see \cite{JinLiZh2})
Let $r_k(s)={\rm sign} \sin(2k\pi s), k=1,2,
\\
\cdots$.
A set of operators $\{M(\lambda):\lambda \in \Lambda\} \subset \mathcal{B}(X)$ is said to be R-bounded if there exists a constant $c>0$ so that for any subset of operators $\{M(\lambda_j)\}_{j=1}^l$, it holds
\begin{equation}\label{St.2}\begin{split}
\int_{0}^{1}\bigg\|\sum_{k=1}^{l}r_k(s)M(\lambda_k)v_k\bigg\|_X\mathrm{d}s
\leq c
\int_{0}^{1}\bigg\|\sum_{k=1}^{l}r_k(s)v_k\bigg\|_X\mathrm{d}s,
\quad \forall v_1,v_2, \cdots, v_l \in X.
\end{split}\end{equation}
\end{def}
The following properties of R-bounded sets are essential to our analysis.
\\
\textbf{Properties}
\begin{itemize}
  \item[(i)] The set $\{\lambda I: \lambda \in \Lambda\}$ is R-bounded;
  \item[(ii)] If $\mathcal{S}, \mathcal{T} \subset \mathcal{B}(X)$ are R-bounded sets, then $\mathcal{S}\cup\mathcal{T}$, $\mathcal{S}+\mathcal{T}$ and $\mathcal{S}\mathcal{T}$ are R-bounded sets.
\end{itemize}
For other aspects of R-boundedness and properties of R-bounded sets, we recommend \cite{JinLiZh2} and the references therein.
\begin{def}\label{def.2}(See \cite{JinLiZh2})
We say an operator $A: D(A)\to X$ is sectorial of angle $\sigma$, provided the following conditions are satisfied,
\begin{itemize}
  \item[{\rm (i)}] $A:D(A)\to X$ is a closed operator and $D(A)$ is dense in $X$;
  \item[{\rm(ii)}] $\sigma(A)\subset \mathbb{C}\setminus\Sigma_\sigma$, where $\sigma(A)$ denotes the spectrum of $A$.
  \item[{\rm(iii)}] The set of operators $\{z(z-A)^{-1}:z\in \Sigma_\sigma\}$ is bounded in $\mathcal{B}(X)$.
\end{itemize}
Further, we say $A$ is R-sectorial of angle $\sigma$ if (i), (ii) and the following condition holds
\\
{\rm(iii')} The set of operators $\{z(z-A)^{-1}:z\in \Sigma_\sigma\}$ is R-bounded in $\mathcal{B}(X)$.
\end{def}
\par
For a sequence $\{M_n\}_{n=0}^\infty$ of operators on $X$, define the generating function
\begin{equation}\label{St.3}\begin{split}
M(\xi):=\sum_{n=0}^{\infty}M_n\xi^n, \quad \forall \xi \in \mathbb{D}'.
\end{split}\end{equation}
\begin{lemma}\label{lem.2}(See \cite{JinLiZh2})
Assume $M:\mathbb{D}'\to \mathcal{B}(X)$ is differentiable such that the set
\begin{equation}\label{St.4}\begin{split}
\{M(\xi):\xi \in \mathbb{D}'\}\cup\{(1-\xi)(1+\xi)M'(\xi):\xi  \in \mathbb{D}'\}
\end{split}\end{equation}
is R-bounded.
Then, we can define a bounded linear operator $\mathcal{M}:\ell^p(X)\to\ell^p(X)$, $1<p<\infty$,
\begin{equation}\label{St.5}\begin{split}
(\mathcal{M}\boldsymbol f)_n:=\sum_{j=0}^{n}M_{n-j}f^j,\quad n=0,1,\cdots,
\quad\text{for any } \boldsymbol f=(f^n)_{n=0}^\infty \in \ell^p(X).
\end{split}\end{equation}
\end{lemma}
\begin{thm}\label{thm.1}(Discrete Gr\"{o}nwall inequality)
Assume $\theta \in (0,\frac{1}{2})$.
For $D_{\tau}^\alpha \psi^{n-\theta}$ defined in (\ref{Se.6}) with weights generated by (\ref{Se.8}), the discrete Gr\"{o}nwall inequality holds:
\par
Let $\alpha \in (0,1)$, $p\in (\frac{1}{\alpha},\infty)$, and a given sequence $(\psi^n)_{n=0}^\infty$ with $\psi^n \in X, n=0,1,2,\cdots$, $\psi^0=0$, satisfy
\begin{equation}\label{St.6}\begin{split}
\|(D_\tau^\alpha \psi^{n-\theta})_{n=1}^m\|_{\ell^p(X)}
\leq \kappa\|(\psi^n)_{n=1}^m\|_{\ell^p(X)}+\eta, \quad
\forall 1\leq m \leq N,
\end{split}\end{equation}
for some positive constants $\kappa$ and $\eta$, then, with sufficiently small $\tau>0$, there holds
\begin{equation}\label{St.7}\begin{split}
\|(\psi^n)_{n=1}^N\|_{\ell^\infty(X)}
+
\|(D_\tau^\alpha \psi^{n-\theta})_{n=1}^N\|_{\ell^p(X)}\leq c\eta,
\end{split}\end{equation}
where the constant $c$ is free of $\tau,N,\eta$, and $\psi^n$, but may depend on $\alpha,p,\kappa,X$ and $T$.
\end{thm}
\begin{proof}
According to the general criterion for discrete fractional Gr\"{o}nwall inequality (Theorem 2.7, \cite{JinLiZh3}), we only need to show that there exists a positive constant $c$ such that the generating function $\omega(\xi)$ satisfies
\begin{equation}\label{St.8}\begin{split}
|\omega(\xi)|\geq \frac{1}{c}|1-\xi|^\alpha \quad \text{and}\quad
|(1-\xi)(1+\xi)\omega'(\xi)|\leq c|\omega(\xi)|, \quad\forall \xi \in \mathbb{D}'.
\end{split}\end{equation}
Actually, for $\omega(\xi)$ defined in (\ref{Se.8}), we can take $c$ as
 \begin{equation}\label{St.9}\begin{split}
c=\sup_{\xi \in \mathbb{D}'}\bigg\{\frac{|1-\xi|^\alpha}{|\omega(\xi)|}, \frac{|(1-\xi)(1+\xi)\omega'(\xi)|}{|\omega(\xi)|}\bigg\}.
\end{split}\end{equation}
Next, we show $c$ is bounded.
Rewrite $\omega(\xi)$ as
\begin{equation}\label{St.9.1}\begin{split}
\omega(\xi)=\mu_0\bigg(\frac{1-\xi}{1+\mu_1\xi}\bigg)^\alpha,
\end{split}\end{equation}
where $\mu_0:=\big(\frac{2\alpha}{\alpha+2\theta}\big)^\alpha$,
$\mu_1:=\frac{\alpha-2\theta}{\alpha+2\theta} \in (-1,1)$.
By letting $\xi=e^{{\rm i}\gamma}$, $\gamma \in (0,2\pi)\setminus\{\pi\}$, ${\rm i}=\sqrt{-1}$, we have
\begin{equation}\label{St.10}\begin{split}
\frac{|1-\xi|^\alpha}{|\omega(\xi)|}
&=\frac{1}{\mu_0}|1+\mu_1 e^{{\rm i}\gamma}|^\alpha
<\frac{1}{\mu_0}\max\{|1+\mu_1|^\alpha,|1-\mu_1|^\alpha\}
<\frac{2}{\mu_0},
\\
\frac{|(1-\xi)(1+\xi)\omega'(\xi)|}{|\omega(\xi)|}
&=\alpha(1+\mu_1)\frac{|1+\xi|}{|1+\mu_1\xi|}
=\alpha(1+\mu_1)\sqrt{\frac{4\cos^2\frac{\gamma}{2}}{(1-\mu_1)^2+4\mu_1\cos^2\frac{\gamma}{2}}}
\\&<
(1+\mu_1)\sqrt{\frac{4}{(1-\mu_1)^2+4\mu_1}}=2.
\end{split}\end{equation}
Note that $\mu_0>\big(\frac{2\alpha}{\alpha+1}\big)^\alpha>\alpha^\alpha\geq e^{-e^{-1}}$.
Thus, taking $c=2e^{e^{-1}}$ completes the proof of the theorem.
\end{proof}
\begin{lemma}\label{lem.3}
Let $\varrho(\xi)=\frac{\omega(\xi)}{1-\theta+\theta\xi}$ with $\theta \in (0,\frac{1}{2})$, then there exists a $\vartheta \in (0,\frac{\pi}{2})$ that depends on $\alpha$ and $\theta$, such that $\varrho(\xi)\in \Sigma_{\vartheta}$ for $\xi \in \mathbb{D}'$.
\end{lemma}
\begin{proof}
Let $\xi=e^{{\rm i}\gamma}$, $\gamma \in (0,\pi)\cup(\pi,2\pi)$.
By (\ref{St.9.1}), we can get
\begin{equation}\label{St.10.1}\begin{split}
\varrho(\xi)=|\varrho(\xi)|e^{{\rm i}\Theta(\gamma,\theta,\alpha)},
\end{split}\end{equation}
where
\begin{equation}\label{St.10.2}\begin{split}
|\varrho(\xi)|&=\mu_0\bigg|2\sin\frac{\gamma}{2}\bigg|^\alpha
(1+\mu_1^2+2\mu_1\cos \gamma)^{-\frac{\alpha}{2}}
\big[(1-\theta+\theta\cos\gamma)^2+\theta^2\sin^2\gamma\big]^{-\frac{1}{2}},
\\
\Theta(\gamma,\theta,\alpha)&:=\frac{\alpha}{2}(\gamma-\pi)-\alpha\beta_1-\beta_2,
\quad
\beta_1:=\arctan\frac{\mu_1\sin\gamma}{1+\mu_1\cos\gamma},
\\
\beta_2&:=\arctan\frac{\theta\sin\gamma}{1-\theta+\theta\cos\gamma}.
\end{split}\end{equation}
Since $\Theta(\gamma,\theta,\alpha)=-\Theta(2\pi-\gamma,\theta,\alpha)$, we can limit the choice of $\gamma \in (0,\pi)$.
A careful derivation shows
\begin{equation}\label{St.10.3}\begin{split}
\frac{\partial \Theta}{\partial \alpha}(\gamma,\theta,\alpha)&
=\frac{1}{2}(\gamma-\pi)-\arctan\bigg(\frac{(\alpha-2\theta)\sin\gamma}{\alpha+2\theta+(\alpha-2\theta)\cos\gamma}\bigg)
\\&\quad-\frac{2\alpha\theta\sin\gamma}{\alpha^2+4\theta^2+(\alpha^2-4\theta^2)\cos\gamma},
\\
\frac{\partial^2 \Theta}{\partial \alpha^2}(\gamma,\theta,\alpha)&
=-\frac{16\theta^3\sin\gamma(1-\cos \gamma)}{\big(\alpha^2+4\theta^2+\alpha^2\cos\gamma-4\theta^2\cos\gamma\big)^2}<0,
\end{split}\end{equation}
and then,
\begin{equation}\label{St.10.4}\begin{split}
\frac{\partial \Theta}{\partial \alpha}(\gamma,\theta,\alpha)
<
\frac{\partial \Theta}{\partial \alpha}(\gamma,\theta,0)=0,
\end{split}\end{equation}
Therefore, we have
\begin{equation}\label{St.10.5}\begin{split}
\Theta(\gamma,\theta,\alpha)\in \big(\Theta(\gamma,\theta,1),\Theta(\gamma,\theta,0)\big).
\end{split}\end{equation}
Note that $\Theta_\theta(\gamma,\theta,0)=-\sin\gamma/\big(1+4\theta(\theta-1)\sin^2(\gamma/2)\big)<0$, as $4\theta(\theta-1) \in (-1,0)$.
Then, we have $\Theta(\gamma,\theta,\alpha)<\Theta(\gamma,\theta,0)<\Theta(\gamma,0,0)=0$.
Further, by a similar analysis, we can show that $\Theta_\gamma(\gamma,\theta,1)>0$, then, it holds
$\Theta(\gamma,\theta,\alpha)>\Theta(0,\theta,1)=-\pi/2$.
For $\gamma \in (0,2\pi)\setminus\{\pi\}$, we thus have $\Theta(\gamma,\theta,\alpha)\in (-\pi/2,\pi/2)$, which completes the proof of the lemma.
\end{proof}
\begin{rem}\label{rem.6}
One can check $\varrho(-1)>0$, and further, for any $\xi \in \{|\xi|\leq 1\}\setminus\{1\}$, there holds $\varrho(\xi) \in \Sigma_\vartheta$.
Actually, by setting $\xi=\delta e^{{\rm i}\gamma}$ with $\delta \in (0,1]$, after a tedious calculation, one can prove that for any fixed $\alpha,\theta$ and $\gamma \in (0,\pi)$, the extreme value of $\Theta$ is always obtained at a point $\xi$ with $|\xi|=1$.
Therefore, we can say $\varrho(\xi)\in \Sigma_\vartheta$ for $|\xi|\leq 1$ and $\xi\neq 1$.
\end{rem}
\begin{thm}\label{thm.2}Assume $v\equiv 0$, $U_h:=W_h$.
For $0<\alpha<1$, $\theta \in (0,\frac{1}{2})$ and sufficiently small $\tau>0$,  the scheme (\ref{Se.14}) is stable with the estimate
\begin{equation}\label{St.11}\begin{split}
\|(U_h^n)_{n=1}^N\|_{\ell^\infty(X)}
+
\|(D_\tau^\alpha U_h^{n-\theta})_{n=1}^N\|_{\ell^p(X)}
\leq c \|(f_h^n)_{n=0}^N\|_{\ell^p(X)}, \quad 1<p<\infty,
\end{split}\end{equation}
where the constant $c$ is independent of $\tau,N,f_h^n$ and $U_h^n$.
\end{thm}
\begin{proof}
We first prove the following maximal $\ell^p$-regularity for the scheme (\ref{Se.14}),
\begin{equation}\label{St.12}\begin{split}
\|(D_\tau^\alpha U_h^{n-\theta})_{n=1}^m\|_{\ell^p(X)}\leq c\|(f^n_h)_{n=0}^N\|_{\ell^p(X)}, \quad \forall 1\leq m \leq N, \quad 1<p<\infty,
\end{split}\end{equation}
where the constant $c$ is independent of $\tau,N,f_h^n$ and $U_h^n$.
Then, (\ref{St.11}) is a direct result of Theorem \ref{thm.1} by choosing $p$ so that $\alpha p>1$.
\par
Multiplying both sides of (\ref{Se.14}) by $\xi^n$ and summing the index from $1$ to $\infty$, we get
\begin{equation}\label{St.13}\begin{split}
&\sum_{n=1}^{\infty}\xi^n D_\tau^\alpha U_h^{n-\theta}
-(1-\theta)\sum_{n=1}^{\infty}\xi^n\Delta_h U_h^{n}
-\theta\sum_{n=1}^{\infty}\xi^n\Delta_h U_h^{n-1}
\\
=&(1-\theta)\sum_{n=1}^{\infty}\xi^nf_h^{n}+\theta\sum_{n=1}^{\infty}\xi^nf_h^{n-1}+(1/2-\theta)\xi f_h^0.
\end{split}\end{equation}
Let $U_h(\xi):=\sum_{n=0}^{\infty}U_h^n\xi^n$, $f_h(\xi):=\sum_{n=0}^{\infty}f_h^n\xi^n$, and recall the fact $U_h^0=0$, we can obtain
\begin{equation}\label{St.14}\begin{split}
\sum_{n=1}^{\infty}\xi^n D_\tau^\alpha U_h^{n-\theta}
&=\tau^{-\alpha}\sum_{n=1}^{\infty}\xi^n\sum_{k=0}^{n}\omega_k U_h^{n-k}
=\tau^{-\alpha}\sum_{n=0}^{\infty}\xi^n\sum_{k=0}^{n}\omega_k U_h^{n-k}
\\
&=\tau^{-\alpha}\omega(\xi)U_h(\xi),
\\
(1-\theta)\sum_{n=1}^{\infty}\xi^n\Delta_h U_h^{n}
+\theta\sum_{n=1}^{\infty}\xi^n&\Delta_h U_h^{n-1}
=(1-\theta)\sum_{n=0}^{\infty}\xi^n\Delta_h U_h^{n}
+\theta\xi\sum_{n=0}^{\infty}\xi^n\Delta_h U_h^{n}
\\
&=(1-\theta+\theta\xi)\Delta_hU_h(\xi),
\\
(1-\theta)\sum_{n=1}^{\infty}\xi^nf_h^{n}+\theta\sum_{n=1}^{\infty}\xi^nf_h^{n-1}
&=(1-\theta)\sum_{n=0}^{\infty}\xi^nf_h^{n}+\theta\xi\sum_{n=0}^{\infty}\xi^nf_h^{n}-(1-\theta)f_h^0
\\&=(1-\theta+\theta\xi)f_h(\xi)-(1-\theta)f_h^0.
\end{split}\end{equation}
Combining (\ref{St.13}) with (\ref{St.14}), we have
\begin{equation}\label{St.15}\begin{split}
\big(\tau^{-\alpha}\varrho(\xi)-\Delta_h\big)U_h(\xi)=f_h(\xi)
+\frac{(1/2-\theta)\xi-(1-\theta)}{1-\theta+\theta\xi}f_h^0,
\end{split}\end{equation}
where $\varrho(\xi)=\frac{\omega(\xi)}{1-\theta+\theta\xi}$.
Since the discrete Laplacian $\Delta_h$ generates an analytic semigroup on $V_h \subset L^2(\Omega)$ \cite{Thomee1}, with the spectrum $\sigma(\Delta_h)$ lies in the negative part of the real line, one can obtain $\Delta_h$ is R-sectorial of angle $\vartheta$\cite{AkrivisLi,JinLiZh2}, where $\vartheta$ is defined in Lemma {\ref{lem.3}}.
Together with $\tau^{-\alpha}\varrho(\xi) \in \Sigma_{\vartheta}$ for $\xi \in \mathbb{D}'$, we know $\tau^{-\alpha}\varrho(\xi)-\Delta_h$ is invertible, which leads to
\begin{equation}\label{St.16}\begin{split}
\sum_{n=1}^{\infty}\xi^n D_\tau^\alpha U_h^{n-\theta}
&=\tau^{-\alpha}\omega(\xi)(\tau^{-\alpha}\varrho(\xi)-\Delta_h)^{-1}\bigg(f_h(\xi)+\frac{(1/2-\theta)\xi-(1-\theta)}{1-\theta+\theta\xi}f_h^0\bigg)
\\&=M(\xi)\big((1-\theta+\theta\xi)f_h(\xi)+[(1/2-\theta)\xi-(1-\theta)] f_h^0\big),
\end{split}\end{equation}
where $M(\xi):=\tau^{-\alpha}\varrho(\xi)(\tau^{-\alpha}\varrho(\xi)-\Delta_h)^{-1}$ is differentiable and R-bounded for $\xi\in\mathbb{D}'$.
By some tedious calculations, we can obtain
\begin{equation}\label{St.17}\begin{split}
M'(\xi)=-A(\xi)M(\xi)+A(\xi)M^2(\xi),
\end{split}\end{equation}
where $A(\xi):=\frac{\alpha(1+\mu_1)}{(1-\xi)(1+\mu_1\xi)}+\frac{\theta}{1-\theta+\theta\xi}$.
Using the \textbf{Properties} of R-boundedness, we know that $(1-\xi)(1+\xi)M'(\xi)$ is R-bounded.
Then, by Lemma \ref{lem.2} and (\ref{St.16}) we have, for $1\leq m \leq N$,
\begin{equation}\label{St.18}\begin{split}
\|(D_\tau^\alpha U_h^{n-\theta})_{n=1}^m\|_{\ell^p(X)}
\leq \|(D_\tau^\alpha U_h^{n-\theta})_{n=1}^\infty\|_{\ell^p(X)}
\leq c\|(\tilde{f}^n_h)_{n=1}^\infty\|_{\ell^p(X)}
=c\|(\tilde{f}^n_h)_{n=1}^N\|_{\ell^p(X)},
\end{split}\end{equation}
where $\tilde{f}_h(\xi):=(1-\theta+\theta\xi)f_h(\xi)+[(1/2-\theta)\xi-(1-\theta)] f_h^0$, which means
\begin{equation}\label{St.19}\begin{split}
\tilde{f}_h^n=
\begin{cases}
  0, & n=0, \\
  (1-\theta)f_h^1+\frac{1}{2} f_h^{0}, & n= 1,\\
  (1-\theta)f_h^n+\theta f_h^{n-1}, & n\geq 2.
\end{cases}
\end{split}\end{equation}
Hence, we have $\|(\tilde{f}^n_h)_{n=1}^N\|_{\ell^p(X)}\leq c \|(f^n_h)_{n=0}^N\|_{\ell^p(X)}$.
By combining the estimate (\ref{St.18}), we complete the proof of the theorem.
\end{proof}
\begin{rem}\label{rem.5}
For the case $f\equiv 0$ with $n\geq 2$, stability estimates can be derived by using the discrete Gr\"{o}nwall type inequality (Theorem \ref{thm.1}) after proving a maximal $\ell^p$-regularity following the same process as Theorem 11 in \cite{JinLiZh2}, which is omitted here for space reasons.
\end{rem}
\section{Sharp error estimates}\label{sec.sha}
In this section, we develop sharp error estimates with respect to the smoothness of the initial data and source term for the scheme (\ref{Se.14}).
Key tools include the Laplace transform and its discrete analogy as well as some analysis for kernel functions.
Throughout this section, we pay special attention on dependence of constants $c$ on the fractional order $\alpha$, aiming to develop $\alpha$-robust (see \cite{ChenStynes}) error estimates for our scheme.
The main results are the following theorems.
\begin{thm}\label{thm.4}
Let $\alpha \in (\alpha_0,1)$ for some $\alpha_0>0$, and $\theta \in (0,1/2)$.
Suppose $u_h(t_n)$ and $U_h^n$ are solutions of (\ref{Se.3}) and (\ref{Se.14.1}), respectively.
For $f\equiv 0$, $n\geq 1$, with sufficiently small $\tau$ and $v_h$ defined by (\ref{Se.4}), we have
\begin{equation}\label{Sh.0.1}\begin{split}
\|u_h(t_n)-U_h^n\|_{L^2(\Omega)}\leq
\begin{cases}
  c t_n^{\alpha-2}\tau^2\|\Delta v\|_{L^2(\Omega)}, & \mbox{if }  v \in D(\Delta),\\
  ct_n^{-2}\tau^2\|v\|_{L^2(\Omega)}, & \mbox{if } v\in L^2(\Omega),
\end{cases}
\end{split}\end{equation}
where $c$ is independent of $\alpha,\tau, n, N$ and $v$, but may depend on $\alpha_0$ and $\theta$.
\end{thm}
\begin{thm}\label{thm.5}
Let $\alpha \in (\alpha_0,1)$ for some $\alpha_0>0$, and $\theta \in (0,1/2)$.
Suppose $u_h(t_n)$ and $U_h^n$ are solutions of (\ref{Se.3}) and (\ref{Se.14.1}), respectively.
For $v\equiv 0$, $f_h=P_hf$ with $f$ satisfying $f\in W^{1,\infty}(0,T;L^2(\Omega))$ and $\int_{0}^{t}(t-s)^{\alpha-1}\|f''(x)\|_{L^2(\Omega)}\mathrm{d}s\in L^{\infty}(0,T)$, it holds, for sufficiently small $\tau$, $n\geq 1$,
\begin{equation}\label{Sh.0.2}\begin{split}
\|u_h(t_n)-U_h^n\|_{L^2(\Omega)}\leq
c\tau^2\bigg(t_n^{\alpha-2}\|f(0)\|_{L^2(\Omega)}+t_n^{\alpha-1}\|f'(0)\|_{L^2(\Omega)}
\\+\int_{0}^{t_n}(t_n-s)^{\alpha-1}\|f''(s)\|_{L^2(\Omega)}\mathrm{d}s\bigg),
\end{split}\end{equation}
where $c$ is independent of $\alpha,\tau, n, N$ and $f$, but may depend on $\alpha_0$ and $\theta$.
\end{thm}
\begin{rem}\label{rem.7}
We adopt the idea in \cite{JinLiZh1} to demonstrate Theorems \ref{thm.4} and \ref{thm.5} in the rest of this section by first proving several key lemmas, which are related closely to our scheme.
The constants $c$ in these lemmas are treated carefully such that they are independent of $\alpha$.
Hence, the estimates in Theorems \ref{thm.4} and \ref{thm.5} are applicable for $\alpha \to 1^-$ which means our estimates are $\alpha$-robust \cite{ChenStynes}.
However, constants $c$ generally depend on the lower bound of $\alpha$, i.e., $\alpha_0$, and may blow up for $\alpha_0\to 0$.
See Lemma \ref{lem.6.6}, Lemma \ref{lem.7} and \textit{Example 3} in Sect.\ref{sec.num}.
We remark that the initial corrections are essential to restore the optimal convergence rate, as for example, only first-order convergence rate can be derived at a fixed time without adding corrections if $v\in D(\Delta)$,  i.e.,
\begin{equation*}\label{Sh.0.3}\begin{split}
\|u_h(t_n)-U_h^n\|_{L^2(\Omega)}\leq
c t_n^{\alpha-1} \tau\|\Delta v\|_{L^2(\Omega)},\quad n\geq 1.
\end{split}\end{equation*}
\end{rem}
\subsection{Laplace transform and solution representations}
Denote by $\widehat{\psi}$ the Laplace transform of $\psi(t)$.
Recall that $\widehat{\partial_t^\alpha \psi}=z^\alpha \widehat{\psi}-z^{\alpha-1}\psi(0)$ and $D_t^\alpha w_h(t)=\partial_t^\alpha w_h(t)$, since $w_h(0)=0$, we have $\widehat{D_t^\alpha w_h}=z^\alpha \widehat{w_h}$.
By taking the Laplace transform of the spacial semidiscrete scheme (\ref{Se.5}), we can get
\begin{equation}\label{Sh.1}\begin{split}
(z^\alpha -\Delta_h) \widehat{w_h}=\widehat{f_h}+z^{-1}\Delta_h v_h.
\end{split}\end{equation}
Then, using inverse Laplace transform, $w_h(t)$ can be represented by
\begin{equation}\label{Sh.2}\begin{split}
w_h(t)=-\frac{1}{2\pi{\rm i}}\int_{\Gamma_{\sigma+\pi/2,\delta}}e^{zt}\big(K(z)\Delta_h v_h+zK(z)\widehat{f_h}(z)\big)\mathrm{d}z,
~
K(z)=-z^{-1}(z^\alpha-\Delta_h)^{-1},
\end{split}\end{equation}
where $K(z)$ is the kernel function.
The contour $\Gamma_{\sigma+\pi/2,\delta}$ (oriented with an increasing imaginary part) is defined by
\begin{equation}\label{Sh.3}\begin{split}
\Gamma_{\sigma+\pi/2,\delta}:=\{z\in\mathbb{C}:|z|=\delta, |\arg z|\leq \sigma+\pi/2, 0<\sigma<\pi/2\}
\\\cup\{z\in\mathbb{C}: z=re^{\pm{\rm i}(\sigma+\frac{\pi}{2})}, r\geq \delta\}.
\end{split}\end{equation}
Since the discrete Laplacian $\Delta_h$ is sectorial of angle $\sigma +\pi/2$, then, $\|(z-\Delta_h)^{-1}\| \leq c|z|^{-1}$, $\forall z \in \Sigma_{\sigma+\pi/2}$.
For $\alpha \in (0,1)$, $z^\alpha=|z|^\alpha e^{{\rm i}\alpha\arg z} \in \Sigma_{\sigma+\pi/2}$, we thus have the resolvent estimate
\begin{equation}\label{Sh.4}\begin{split}
\|(z^\alpha-\Delta_h)^{-1}\| \leq c|z|^{-\alpha},\quad \forall z\in \Sigma_{\sigma+\pi/2}.
\end{split}\end{equation}
\par
Next lemma gives a discrete analogy of the continuous solution representation (\ref{Sh.2}).
\begin{thm}\label{thm.6}
Let $g_h^n:=f_h^n-f_h^0$.
For $\alpha \in (0,1)$ and $\theta \in (0,1/2)$, there exists a $\sigma_0 \in (0,\pi/2)$  which is independent of $\alpha$ and $\tau$, so that for any $\sigma, \delta\in (0,\sigma_0)$, the solution $W_h^n$ of the scheme (\ref{Se.14}) can be expressed by
\end{thm}
\begin{equation}\label{Sh.5}\begin{split}
W_h^n=\frac{1}{2\pi{\rm i}}\int_{\Gamma_{\sigma+\pi/2,\delta}^\tau}e^{zt_n}\big[
\mu(e^{-z\tau})K(\beta_\tau(e^{-z\tau}))(-\Delta_hv_h-f_h^0)
\\-\beta_\tau(e^{-z\tau})K(\beta_\tau(e^{-z\tau}))g_h(e^{-z\tau})\tau
\big]\mathrm{d}z,
\end{split}\end{equation}
where $K(z)$ denotes the kernel function, and $\Gamma_{\sigma+\pi/2,\delta}^\tau$ is part of the contour $\Gamma_{\sigma+\pi/2,\delta}$ defined by $\Gamma_{\sigma+\pi/2,\delta}^\tau:=\{z\in \Gamma_{\sigma+\pi/2,\delta}:|\Im(z)|\leq \pi/\tau\}$.
$\beta_\tau(\xi)$ and $\mu(\xi)$ are given as follows
\begin{equation}\label{Sh.6}\begin{split}
\beta_\tau(\xi)&:=\frac{1}{\tau}\bigg[\frac{\omega(\xi)}{1-\theta+\theta\xi}\bigg]^\frac{1}{\alpha}
=\mu_0^{\frac{1}{\alpha}}\frac{1-\xi}{\tau(1+\mu_1\xi)}(1-\theta+\theta\xi)^{-\frac{1}{\alpha}},
\\
\mu(\xi)&:=\frac{\mu_0^{\frac{1}{\alpha}}(3/2-\theta)\xi(1-\mu_2\xi)}{(1+\mu_1\xi)(1-\theta+\theta\xi)^{1+\frac{1}{\alpha}}},
\end{split}\end{equation}
where $\mu_0=\big(\frac{2\alpha}{\alpha+2\theta}\big)^\alpha, \mu_1=\frac{\alpha-2\theta}{\alpha+2\theta}, \mu_2=\frac{1-2\theta}{3-2\theta}$.
\par
Before proving Theorem \ref{thm.6}, we need to demonstrate some mapping properties of $\varrho(\xi)$ and $\beta_\tau^\alpha(e^{-z\tau})$ first.
\begin{lemma}\label{lem.5}
Given $\alpha \in (0,1)$, $\vartheta \in (\pi/2,\pi)$ and $\theta \in (0,1/2)$.
Let $\varrho(\xi)=\frac{\omega(\xi)}{1-\theta+\theta\xi}$.
There exists a $\sigma_0 \in (0,\pi/2)$, which is independent of $\alpha$ (but may depend on $\theta$ and $\vartheta$), such that for any $\xi \in \{\xi \in \mathbb{C}\setminus \overline{\Sigma}_{\pi-\sigma_0}: |\xi|\leq 1+\sigma_0\}$, it holds $\varrho(\xi) \in \Sigma_{\vartheta}$.
\end{lemma}
\begin{proof}
Give $\theta \in (0,1/2)$, we take $\sigma_1:=\min\{\frac{1-2\theta}{\theta},\frac{4\theta}{1-2\theta},\frac{1}{2}\}$.
Let $\xi=\delta e^{{\rm i}\gamma}$, $\delta\leq 1+\sigma_1$ and $\gamma\in (\pi-\sigma_1,\pi)\cup[-\pi,-\pi+\sigma_1)$.
By (\ref{St.10.1}), we can obtain
\begin{equation}\label{Sh.6.1}\begin{split}
\Theta=\Theta(\delta,\gamma,\theta,\alpha)=\alpha\beta_0-\alpha\beta_1-\beta_2,
\end{split}\end{equation}
where
\begin{equation}\label{Sh.6.2}\begin{split}
\beta_0=-\arctan\frac{\delta\sin\gamma}{1-\delta\cos\gamma},~
\beta_1=\arctan\frac{\mu_1\delta\sin\gamma}{1+\mu_1\delta\cos\gamma},~
\beta_2=\arctan\frac{\theta\delta\sin\gamma}{1-\theta+\theta\delta\cos\gamma}.
\end{split}\end{equation}
For simplicity, define the function $S(\lambda,\delta):=\arctan\frac{\delta\sin\gamma}{\lambda+\delta\cos\gamma}$.
Then, $\beta_0=S(-1,\delta)$, $\beta_1=S(\frac{1}{\mu_1},\delta)$ and $\beta_2=S(\frac{1-\theta}{\theta},\delta)$.
Since $\Theta(\delta,-\gamma,\theta,\alpha)=-\Theta(\delta,\gamma,\theta,\alpha)$, it suffices to analyze the case $\gamma \in (\pi-\sigma_1,\pi)$.
\par
\textit{Step 1. $\Theta<0$.}
\par
Obviously, $\beta_0<0$.
For $\beta_2$, we have
\begin{equation}\label{Sh.6.2.1}\begin{split}
\frac{1-\theta}{\theta}+\delta\cos\gamma>\frac{1-\theta}{\theta}-\delta\geq \frac{1-2\theta}{\theta}-\sigma_1\geq 0,\quad
\beta_2=\arctan\frac{\delta\sin\gamma}{\frac{1-\theta}{\theta}+\delta\cos\gamma}>0.
\end{split}\end{equation}

For $\beta_1$, if $\alpha=2\theta$, then $\beta_1=0$, $\Theta=\alpha \beta_0-\beta_2<0$.
If $\alpha>2\theta$, we have
\begin{equation}\label{Sh.6.3}\begin{split}
&\frac{1}{\mu_1}+\delta\cos\gamma > \frac{\alpha+2\theta}{\alpha-2\theta}-\delta > \frac{1+2\theta}{1-2\theta}-\delta
\geq \frac{4\theta}{1-2\theta}-\sigma_1\geq 0,
\\
&\beta_1=\arctan\frac{\delta\sin\gamma}{\frac{1}{\mu_1}+\delta\cos\gamma}>0,
\end{split}\end{equation}
with which there holds $\Theta=\alpha\beta_0-\alpha\beta_1-\beta_2<0$.
If $\alpha<2\theta$, one can check $S_\lambda(\lambda,\delta)<0$ and $\frac{1}{\mu_1}<-1$, then $\Theta=\alpha\big(S(-1,\delta)-S(\frac{1}{\mu_1},\delta)\big)-\beta_2<0$.
\par
\textit{Step 2. $\Theta$ is a decreasing function with respect to $\alpha$}.
\par
By a direct calculation, we have
\begin{equation}\label{Sh.6.4}\begin{split}
\frac{\partial \Theta}{\partial \alpha}=S(-1,\delta)-S(1/\mu_1,\delta)-\frac{4\alpha\delta\theta\sin\gamma}{\big[(\alpha-2\theta)\delta+(\alpha+2\theta)\cos\gamma\big]^2+(\alpha+2\theta)^2\sin^2\gamma}
\end{split}\end{equation}
Appealing to a similar analysis in \textit{Step 2.}, we can get $\frac{\partial \Theta}{\partial \alpha}<0$.
Combining with $\Theta<0$, we conclude $\Upsilon(\delta,\gamma,\theta):=\Theta(\delta,\gamma,\theta,1)<\Theta(\delta,\gamma,\theta,\alpha)<0$.
\par
\textit{Step 3. $\Upsilon$ is a decreasing function with respect to $\delta$}.
\par
One can get
\begin{equation}\label{Sh.6.5}\begin{split}
\frac{\partial \Upsilon}{\partial \delta}
=S_\delta(-1,\delta)-S_\delta(\frac{1+2\theta}{1-2\theta},\delta)-S_\delta(\frac{1-\theta}{\theta},\delta).
\end{split}\end{equation}
Note that $S_\delta(\lambda,\delta)=\lambda\sin\gamma/\big[(\delta+\lambda\cos\gamma)^2+(\lambda\sin\gamma)^2\big]$, we thus have $\frac{\partial \Upsilon}{\partial \delta}<0$.
\par
Now, we note for arbitrary given $\theta \in (0,1/2)$, $\gamma \in (\pi-\sigma_1,\pi)\cup (-\pi,-\pi+\sigma_1)$,  it holds $\Upsilon(1,\gamma,\theta) \in (-\pi/2,\pi/2)$ by Lemma \ref{lem.3}.
Then, for a given $\vartheta \in (\pi/2,\pi)$ which is close to $\pi/2$, by the continuity of $\Upsilon$ with respect to $\delta$, we can take a small $\sigma_2$ (may depend on $\vartheta$) and let $\sigma_0=\min\{\sigma_1,\sigma_2\}$,
such that $\Upsilon \in (-\vartheta,\vartheta)$  for $\delta \leq 1+\sigma_0$, $\gamma \in (\pi-\sigma_0,\pi)\cup (-\pi,-\pi+\sigma_0)$.
The proof of the lemma is completed.
\end{proof}
\begin{lemma}\label{lem.6}
Given $\alpha \in (0,1)$, $\theta \in (0,1/2)$, $\vartheta \in (\pi/2,\pi)$ and $\tau>0$.
Then, there exists a $\sigma_0 \in (0,\pi/2)$, which is independent of $\alpha$ and $\tau$ (but may depend on $\theta$ and $\vartheta$), such that for any $z \in \mathbb{S}^\tau_{\sigma_0}:=\{z=x+{\rm i}y\in\Sigma_{\sigma_0+\pi/2}:|y|<(\pi+\sigma_0)/\tau\}$, it holds $\beta^\alpha_\tau(e^{-z\tau}) \in \Sigma_\vartheta$.
\end{lemma}
\begin{proof}
We choose $\sigma_0$ as stated in Lemma \ref{lem.5} and take two steps in the following arguments: (i) $z\in \overline{\Sigma}_{\pi/2}\setminus \{0\}$, or (ii) $z\in \mathbb{S}^\tau_{\sigma_0}\setminus \overline{\Sigma}_{\pi/2}$.
\par
For $z\in \overline{\Sigma}_{\pi/2}\setminus \{0\}$, let $z=x+{\rm i}y$ where $x\geq 0$ and $x^2+y^2\neq 0$.
Then, $|e^{-\tau z}|=|e^{-\tau x}e^{-{\rm i}\tau y}|=e^{-\tau x}\leq 1$ and $e^{-z\tau}\neq 1$.
Using Lemma \ref{lem.3} and Remark \ref{rem.6}, we have $\varrho(e^{-z\tau})\in \Sigma_{\pi/2}$ and $\beta^\alpha_\tau(e^{-z\tau})=\tau^{-\alpha}\varrho(e^{z\tau})\in \Sigma_{\pi/2}$.
\par
For $z\in \mathbb{S}^\tau_{\sigma_0}\setminus \overline{\Sigma}_{\pi/2}$, considering the symmetry we can take $z=|z|e^{{\rm i}\arg z}$ with $\arg z \in (\pi/2,\pi/2+\sigma_0)$.
Then,
\begin{equation}\label{Sh.6.1.1}\begin{split}
e^{-z\tau}=e^{-\tau|z|\cos(\arg z)}e^{-{\rm i}\tau|z|\sin(\arg z)}.
\end{split}\end{equation}
Since $z\in \mathbb{S}^\tau_{\sigma_0}$, it holds that $\tau|z|\sin(\arg z)< \pi+\sigma_0$, $\arg(e^{-z\tau})>\pi-\sigma_0$.
Similarly, for $\arg z \in (-\pi/2-\sigma_0,-\pi/2)$, we have $\arg(e^{-z\tau})<-\pi+\sigma_0$.
Moreover, one can check that, if $|\arg z|\to \frac{\pi}{2}$,
\begin{equation}\label{Sh.6.1.2}\begin{split}
|e^{-z\tau}|=e^{-\tau|z|\cos(\arg z)}
> e^{-(\pi+\sigma_0)\cos(\arg z)/|\sin (\arg z)|}\to 1.
\end{split}\end{equation}
Thus, we can take a small $\sigma_0$ such that for any $z\in \mathbb{S}^\tau_{\sigma_0}\setminus \overline{\Sigma}_{\pi/2}$, with $|\arg z|$ sufficiently close to $\frac{\pi}{2}$, there holds $e^{-z\tau} \in \{\xi \in \mathbb{C}\setminus \overline{\Sigma}_{\pi-\sigma_0}: |\xi|\leq 1+\sigma_0\}$.
By Lemma \ref{lem.5} and $\beta^\alpha_\tau(e^{-z\tau})=\tau^{-\alpha}\varrho(e^{z\tau})$, we complete the proof of the lemma.
\end{proof}
\begin{proof}[The proof of Theorem \ref{thm.6}]
Multiply (\ref{Se.14}) by $\xi^n$ and sum the index from $n=1$ to $\infty$ to arrive at
\begin{equation}\label{Sh.7}\begin{split}
&\sum_{n=1}^{\infty}\xi^n D_\tau^\alpha W_h^{n-\theta}-\sum_{n=1}^{\infty}\xi^n\Delta_h W_h^{n-\theta}
\\=&\sum_{n=1}^{\infty}\xi^n g_h^{n-\theta}
+(1/2-\theta)(\Delta_h v_h+f^0_h)\xi
+\sum_{n=1}^{\infty}(\Delta_h v_h+f_h^0)\xi^n.
\end{split}\end{equation}
Then, by a derivation similar to Theorem \ref{thm.2}, we can obtain
\begin{equation}\label{Sh.8}\begin{split}
\big(\beta_\tau^\alpha(\xi)-\Delta_h\big)W_h(\xi)=\kappa(\xi)(\Delta_h v_h+f_h^0)+g_h(\xi),
\end{split}\end{equation}
where $\kappa(\xi):=\xi\big(\frac{1}{1-\xi}+\frac{1}{2}-\theta\big)/(1-\theta+\theta\xi)$.
Since $\beta_\tau^\alpha(\xi)-\Delta_h$ is invertible for any $\xi \in \{|\xi|\leq 1, \xi \neq 1\}$ (Remark \ref{rem.6} and (\ref{Sh.4})), we get
\begin{equation}\label{Sh.9}\begin{split}
W_h(\xi)=\big(\beta_\tau^\alpha(\xi)-\Delta_h\big)^{-1}\big(\kappa(\xi)(\Delta_h v_h+f_h^0)+g_h(\xi)\big).
\end{split}\end{equation}
One can check $W_h(\xi)$ is analytic in the circle $\{\xi:|\xi|<\varepsilon\}$ for small $\varepsilon>0$.
Then, by Cauchy's integral formula, we obtain
\begin{equation}\label{Sh.10}\begin{split}
W_h^n=\frac{W_h^{(n)}(0)}{n!}
=\frac{1}{2\pi{\rm i}}\int_{|\xi|=\varepsilon}\frac{W_h(\xi)}{\xi^{n+1}}\mathrm{d}\xi.
\end{split}\end{equation}
Let $\xi=e^{-z\tau}$, $z \in \Gamma^\tau_\varepsilon:=\{z=-\ln\varepsilon+{\rm i}y: y\in \mathbb{R}, |y|\leq \pi/\tau\}$ with an increasing imaginary part, we get
\begin{equation}\label{Sh.11}\begin{split}
W_h^n=\frac{\tau}{2\pi{\rm i}}\int_{\Gamma_\varepsilon^\tau}e^{zt_n}W_h(e^{-z\tau})\mathrm{d}z.
\end{split}\end{equation}
Using Lemma \ref{lem.6}, we can conclude $\big(\beta_\tau^\alpha(e^{-z\tau})-\Delta_h\big)^{-1}$ is analytic for $z\in \overline{\mathbb{S}}$, where $\mathbb{S}$ is a region enclosed by contours $\Gamma_{\sigma+\pi/2,\delta}^\tau$, $\Gamma_\varepsilon^\tau$, $\Gamma^\tau_{\pm}:=\mathbb{R}\pm {\rm i}\pi/\tau$ (oriented from left to right), for some small positive parameters $\sigma, \delta \in (0,\sigma_0)$ where $\sigma_0>0$ is independent of $\alpha$ and $\tau$.
Then, $e^{zt_n}W_h(e^{-z\tau})$ is analytic for $z\in \overline{\mathbb{S}}$ as both $\kappa(e^{-z\tau})$ and $g_h(e^{-z\tau})$ are analytic functions for $z\in \overline{\mathbb{S}}$.
We thus can derive by Cauthy's theorem that
\begin{equation}\label{Sh.12}\begin{split}
W_h^n&=\frac{\tau}{2\pi{\rm i}}\int_{\Gamma_\varepsilon^\tau}e^{zt_n}W_h(e^{-z\tau})\mathrm{d}z
=\frac{\tau}{2\pi{\rm i}}\int_{\Gamma_{\sigma+\pi/2,\delta}^\tau}e^{zt_n}W_h(e^{-z\tau})\mathrm{d}z
\\
&\quad-\frac{\tau}{2\pi{\rm i}}\int_{\Gamma_{-}^\tau}e^{zt_n}W_h(e^{-z\tau})\mathrm{d}z
+\frac{\tau}{2\pi{\rm i}}\int_{\Gamma_{+}^\tau}e^{zt_n}W_h(e^{-z\tau})\mathrm{d}z
\\
&=\frac{\tau}{2\pi{\rm i}}\int_{\Gamma_{\sigma+\pi/2,\delta}^\tau}e^{zt_n}W_h(e^{-z\tau})\mathrm{d}z,
\end{split}\end{equation}
where we have used the fact $e^{zt_n}W_h(e^{-z\tau})$ takes the same value on the two lines $\Gamma_{\pm}^\tau$.
By $K(z)$ defined in (\ref{Sh.2}) and letting $\mu(\xi)=\tau\beta_\tau(\xi)\kappa(\xi)$, we complete the proof of the theorem.
\end{proof}
\begin{rem}\label{rem.8}
We remark that if $\theta=\frac{1}{2}$, Lemma \ref{lem.5} is no longer true since by the argument $\sigma_0=\min\{\sigma_1,\sigma_2\}$ where $\sigma_1=0$, we thus obtain $\sigma_0=0$ which contradicts the assumption $\sigma_0 \in (0,\pi/2)$.
Moreover, the function $\kappa(e^{-z\tau})$ in the above theorem will be singular at points $z=\pm {\rm i}\pi/\tau$.
\end{rem}
\subsection{Key lemmas on kernel functions}
We develop some key estimates for functions $\beta_\tau(\xi)$ and $\mu(\xi)$.
In order to develop $\alpha$-robust analysis, we require the constant $c$ is independent of $\alpha$.
\begin{lemma}\label{lem.6.6}
Given $\theta \in (0,1/2)$, $\alpha_0,\tau \in (0,1)$.
There exists a $\sigma_1 =\sigma_1(\theta,\alpha_0)\in (0,\pi/2)$, such that for any $\alpha \in (\alpha_0,1)$ and $\sigma,\delta \in (0,\sigma_1]$, there hold
\begin{equation}\label{Sh.12.1.1}\begin{split}
0<c_0\leq|1+\mu_1e^{-z\tau}|^{\widetilde{\beta}}\leq c_1,\quad
0<c_0\leq|1-\theta+\theta e^{-z\tau}|^{\widetilde{\alpha}}\leq c_1,
\quad \forall z \in \Gamma_{\sigma+\pi/2,\delta}^\tau,
\end{split}\end{equation}
where $\widetilde{\alpha}$ take values from $\{1, 1/\alpha, 1+1/\alpha\}$, $\widetilde{\beta}$ is $1$ or $\alpha$, and constants $c_0$ and $c_1$ depend only on $\theta$ and $\alpha_0$.
\begin{proof}
For any $z\in \Gamma_{\sigma+\pi/2,\delta}$ with small $\sigma,\delta \in(0,\pi/2)$, since $\tau<1$, we have $\pi/\tau>\pi>\delta$.
Let $z=x+{\rm i}y$, $x\in [-\frac{\pi}{\tau}\tan \sigma,\delta]$, $y\in [-\pi/\tau,\pi/\tau]$.
Then, $e^{-z\tau}=e^{-\tau x}e^{-{\rm i}\tau y}$, $|e^{-z\tau}|=e^{-\tau x} \leq e^{\pi\tan \sigma}$.
For given $\theta$ and $\alpha_0$, $\mu_1$ can be bounded by a constant $M(\alpha_0,\theta)<1$.
By taking $e^{\pi \tan \sigma_1}<\frac{1}{M}$, i.e., $\sigma_1<\arctan(\frac{1}{\pi}\ln \frac{1}{M})$, we have $|\mu_1 e^{-z\tau}|=e^{-x\tau}|\mu_1|\leq M e^{\pi\tan \sigma_1} <1$.
Using the triangle inequality, we obtain $|1+\mu_1 e^{-z\tau}|\leq 1+|\mu_1 e^{-z\tau}|<2$ and $|1+\mu_1 e^{-z\tau}|\geq 1-|\mu_1 e^{-z\tau}|\geq 1-M e^{\pi\tan \sigma_1}>0$.
We can similarly analyze the term $|1-\theta+\theta e^{-z\tau}|$ since $|1-\theta+\theta e^{-z\tau}|=(1-\theta)|1+\frac{\theta}{1-\theta} e^{-z\tau}|$ and $\frac{\theta}{1-\theta} \in (0,1)$.
For fixed $\theta$, we take $\sigma_1=\sigma_1(\theta)$ satisfying $e^{\pi\tan \sigma_1}<\frac{\theta}{1-\theta}$, then
$|1+\frac{\theta}{1-\theta} e^{-z\tau}|\leq 2$ and $|1+\frac{\theta}{1-\theta} e^{-z\tau}| \geq 1-\frac{\theta}{1-\theta}|e^{-z\tau}|\geq 1-\frac{\theta}{1-\theta}e^{\pi\tan\sigma_1}>0$.
The proof of the lemma is completed.
\end{proof}
\end{lemma}
\begin{lemma}\label{lem.7}
Given $\theta \in (0,1/2)$, $\alpha_0,\tau_0 \in (0,1)$ and $\tau<\tau_0$.
Then, there exists a $\sigma_1 =\sigma_1(\theta,\alpha_0)>0$, such that for any $\alpha \in (\alpha_0,1)$, $\sigma,\delta \in (0,\sigma_1]$, and $z\in \Gamma_{\sigma+\pi/2,\delta}^\tau$, there hold
\begin{equation}\label{Sh.13}\begin{split}
&({\rm i})~ |\mu(e^{-z\tau})-1|\leq c\tau^2|z|^2,\quad
({\rm ii})~ |\beta_\tau(e^{-z\tau})-z|\leq c\tau^2|z|^3,\\
&({\rm iii})~ |\beta^\alpha_\tau(e^{-z\tau})-z^\alpha|\leq c\tau^2|z|^{2+\alpha},\quad
({\rm iv})~ c_0|z|\leq |\beta_\tau(e^{-z\tau})|\leq c_1|z|,
\end{split}\end{equation}
where constants $c,c_0$ and $c_1$ depend only on $\theta,\alpha_0$ and $\tau_0$.
\end{lemma}
\begin{proof}
Take $\sigma_1(\theta,\alpha_0)$ such that (\ref{Sh.12.1.1}) hold.
Let $\epsilon:=\min\{\epsilon_0,\ln\frac{1-\theta}{\theta}\}$ where $\epsilon_0$ is positive and is sufficiently small.
If $|z|\tau<\epsilon$, using the Taylor expansion, we have
\begin{equation}\label{Sh.13.1}\begin{split}
e^{-z\tau}&=1-z\tau+\frac{1}{2}z^2\tau^2+\sum_{k=3}^{\infty}\frac{(-1)^k}{k!}z^k\tau^k,
\\
(1-\theta+\theta e^{-z\tau})^{\widetilde{\alpha}}
&=1+\widetilde{\alpha}(-\theta+\theta e^{-z\tau})
+\frac{1}{2}\widetilde{\alpha}(\widetilde{\alpha}-1)(-\theta+\theta e^{-z\tau})^2
\\&\quad
+\sum_{k=3}^{\infty}
\begin{pmatrix}
  \widetilde{\alpha} \\
  k
\end{pmatrix}(-\theta+\theta e^{-z\tau})^k
\\&=1-\widetilde{\alpha}\theta z\tau
+\frac{1}{2}\widetilde{\alpha}\theta[1+(\widetilde{\alpha}-1)\theta](z\tau)^2
+\sum_{k=3}^{\infty}A_k(\widetilde{\alpha},\theta)(z\tau)^k,
\end{split}\end{equation}
where $\widetilde{\alpha}$ takes $1/\alpha$ or $1+1/\alpha$.
We have used the analyticity of $(1-\theta+\theta e^{-z\tau})^{\widetilde{\alpha}}$ with respect to $z$ in (\ref{Sh.13.1}), since for $z$ satisfying $|z|\tau<\ln \frac{1-\theta}{\theta}$, it holds
\begin{equation}\label{Sh.13.1.1}\begin{split}
|e^{-z\tau}|\leq e^{|z|\tau}<\frac{1-\theta}{\theta},~
|1-\theta+\theta e^{-z\tau}|=\theta\bigg|\frac{1-\theta}{\theta}+e^{-z\tau}\bigg|
\geq \theta \bigg(\frac{1-\theta}{\theta}-|e^{-z\tau}|\bigg)>0.
\end{split}\end{equation}
We note that $A_k(\widetilde{\alpha},\theta)$ can be bounded by some constant depending only on $\alpha_0$ and $\theta$.
Combining (\ref{Sh.6}) with the above expansions and using Lemma \ref{lem.6.6}, we thus obtain
\begin{equation}\label{Sh.13.2}\begin{split}
|\mu(e^{-z\tau})-1|
&\leq c(\theta,\alpha_0)
\big|\mu_0^{\frac{1}{\alpha}}(3/2-\theta)e^{-z\tau}(1-\mu_2e^{-z\tau})
\\&\quad-(1+\mu_1e^{-z\tau})(1-\theta+\theta e^{-z\tau})^{1+\frac{1}{\alpha}}\big|
\\&\leq
c(\theta,\alpha_0)\bigg(c(\theta,\alpha)\tau^2|z|^2
+\tau^2|z|^2\sum_{k=3}^{\infty}|\widetilde{A}_k(\alpha,\theta)|(\tau|z|)^{k-2}\bigg),
\end{split}\end{equation}
where $\widetilde{A}_k(\alpha,\theta)$ can be bounded by some constant depending only on $\alpha_0$ and $\theta$.
The constant $c(\theta,\alpha)$ is
\begin{equation}\label{Sh.13.3}\begin{split}
c(\theta,\alpha)=\frac{|\theta^2-(1-\theta)(\theta+\alpha^2)|}{\alpha(\alpha+2\theta)}\leq c(\theta,\alpha_0).
\end{split}\end{equation}
Therefore, we have $|\mu(e^{-z\tau})-1| \leq c(\theta,\alpha_0)\tau^2|z|^2$ when $|z|\tau<\epsilon$.
Similarly, for $\beta_\tau(e^{-z\tau})$ and $\beta^\alpha_\tau(e^{-z\tau})$, we have
\begin{equation}\label{Sh.13.4}\begin{split}
&\quad|\tau\beta_\tau(e^{-z\tau})-z\tau|
\\&\leq
c(\theta,\alpha_0)\big|\mu_0^{\frac{1}{\alpha}}(1-e^{-z\tau})-z\tau(1+\mu_1e^{-z\tau})(1-\theta+\theta e^{-z\tau})^\frac{1}{\alpha}\big|
\\&\leq
c(\theta,\alpha_0)\bigg(\bar{c}_0(\theta,\alpha)\tau^3|z|^3
+\tau^3|z|^3\sum_{k=4}^{\infty}|\widetilde{B}_k(\alpha,\theta)|(\tau|z|)^{k-3}\bigg),
\\
&\quad|\tau^\alpha\beta^\alpha_\tau(e^{-z\tau})-z^\alpha\tau^\alpha|\leq
c(\theta,\alpha_0)
|z|^\alpha\tau^\alpha\bigg|\mu_0\bigg(\frac{1-e^{-z\tau}}{z\tau}\bigg)^\alpha
\\&\quad-(1+\mu_1e^{-z\tau})^\alpha(1-\theta+\theta e^{-z\tau})\bigg|
\\&\leq
c(\theta,\alpha_0)|z|^{\alpha}\tau^{\alpha}\bigg(\bar{c}_1(\theta,\alpha)\tau^2|z|^2+
\tau^2|z|^2\sum_{k=3}^{\infty}|\widetilde{C}_k(\alpha,\theta)|(\tau|z|)^{k-2}\bigg),
\end{split}\end{equation}
where $\widetilde{B}_k(\alpha,\theta), \widetilde{C}_k(\alpha,\theta)$ can be bounded by some constants depending only on $\alpha_0$ and $\theta$, and
\begin{equation}\label{Sh.13.5}\begin{split}
\bar{c}_0(\theta,\alpha)&=\frac{|\alpha^2+6\alpha\theta-6\alpha\theta^2-6\theta^2|}{6\alpha(\alpha+2\theta)}\leq c(\theta,\alpha_0),
\\
\bar{c}_1(\theta,\alpha)&=\frac{|\alpha^2+6\alpha\theta-6\alpha\theta^2-6\theta^2|}{12\alpha}\bigg(\frac{2\alpha}{\alpha+2\theta}\bigg)^\alpha
\leq c(\theta,\alpha_0).
\end{split}\end{equation}
Otherwise, if $|z|\tau\geq\epsilon$, by Lemma \ref{lem.6.6} and the fact $e^{-z\tau}$ can be bounded by a constant depending only on $\theta$ and $\alpha_0$, we have the following estimates
\begin{equation}\label{Sh.13.100}\begin{split}
|\mu(e^{-z\tau})-1|&\leq |\mu(e^{-z\tau})|+1
\\&\leq c(\theta,\alpha_0) \frac{2\alpha}{\alpha+2\theta}(3/2-\theta)|e^{-z\tau}||1-\mu_2 e^{-z\tau}|+1
\\
&\leq c(\theta,\alpha_0)\tau^2|z|^2,
\\
|\tau\beta_\tau(e^{-z\tau})-z\tau|
&\leq |\tau\beta_\tau(e^{-z\tau})|+\tau|z|
\leq c(\theta,\alpha_0) \frac{2\alpha}{\alpha+2\theta}|1-e^{-z\tau}|  +\tau|z|
\\&\leq c(\theta,\alpha_0,\tau_0)\tau^3|z|^3,
\\
|\tau^\alpha\beta^\alpha_\tau(e^{-z\tau})-z^\alpha\tau^\alpha|
&\leq |\tau^\alpha\beta^\alpha_\tau(e^{-z\tau})|+\tau^\alpha|z|^\alpha
\\&\leq  c(\theta,\alpha_0)\bigg(\frac{2\alpha}{\alpha+2\theta}\bigg)^\alpha |1-e^{-z\tau}|^\alpha  +\tau^\alpha|z|^\alpha
\\&\leq c(\theta,\alpha_0,\tau_0)\tau^{2+\alpha}|z|^{2+\alpha}.
\end{split}\end{equation}
We thus have proved (i)-(iii).
For (iv), by (ii) we get $|\beta_\tau(e^{-z\tau})/z-1|\leq c\tau^2|z|^2
\leq c\pi^2/\cos^2 \sigma_1
\leq c(\theta,\alpha_0,\tau_0)$, $|\beta_\tau(e^{-z\tau})/z| \leq |\beta_\tau(e^{-z\tau})/z-1|+1\leq c(\theta,\alpha_0,\tau_0)$.
Considering $c_0'|z|<\big|\frac{1-e^{-z\tau}}{\tau}\big|<c_1'|z|$(see \cite{JinLiZh1}) where $c_0'$ and $c_1'$ are independent of $\alpha,\theta$, by Lemma \ref{lem.6.6} again, we complete the proof of (iv).
\end{proof}
\par
Now, based on the above lemmas, we can prove Theorems \ref{thm.4} and \ref{thm.5} in exactly the same way as in \cite{JinLiZh1}, which is omitted here for space reasons.
\section{Fast algorithms}\label{sec.fast}
We develop two fast algorithms for the fully discrete scheme (\ref{Se.9}) in this section based on the algorithms in \cite{Schadle} and \cite{Zeng1}, respectively.
To start with, we reformulate the generating function (\ref{Se.8}) into a contour integral by two different strategies,
\begin{equation}\label{Fa.1}\begin{split}
(\textbf{I})&\quad\tau^{-\alpha}\omega(\xi)=\tau^{-\alpha}\sum_{n=0}^{\infty}\omega_n\xi^n
=\frac{\tau}{2\pi{\rm i}}\int_{\mathcal{C}}\bigg[\frac{1-\xi}{\frac{1}{2}(1+\xi)+\frac{\theta}{\alpha}(1-\xi)}-\lambda\tau\bigg]^{-1}F^{(1)}(\lambda)\mathrm{d}\lambda,
\\
(\textbf{II})&\quad\tau^{-\alpha}\omega(\xi)=\tau^{-\alpha}\sum_{n=0}^{\infty}\omega_n\xi^n
=F^{(2)}\bigg(\frac{1-\xi}{\tau}\bigg)
=\frac{\tau}{2\pi{\rm i}}\int_{\mathcal{C}}(1-\xi-\lambda\tau)^{-1}F^{(2)}(\lambda)\mathrm{d}\lambda,
\end{split}\end{equation}
where $F^{(1)}(\lambda)=\lambda^\alpha$, $F^{(2)}(\lambda)=[\frac{1}{\lambda}+(\frac{\theta}{\alpha}-\frac{1}{2})\tau]^{-\alpha}$ and $\mathcal{C}$ is a Talbot contour \cite{Zeng1} defined by
\begin{equation}\label{Fa.2}\begin{split}
(-\pi,\pi)\to \mathcal{C}: \vartheta \mapsto \lambda(\vartheta,\varsigma)=\varsigma\big((\vartheta \cot(\vartheta)+{\rm i} \kappa\vartheta)\nu+\iota\big),
\\
\kappa=0.5653, ~\nu=0.6443, ~\iota=-0.4814.
\end{split}\end{equation}
Let $e^{(1)}_n(z)$ and $e^{(2)}_n(z)$ satisfy
\begin{equation}\label{Fa.3}\begin{split}
\sum_{n=0}^{\infty}e^{(1)}_n(z)\xi^n
&=\bigg[\frac{1-\xi}{\frac{1}{2}(1+\xi)+\frac{\theta}{\alpha}(1-\xi)}-z\bigg]^{-1},
\\
\sum_{n=0}^{\infty}e^{(2)}_n(z)\xi^n
&=(1-\xi-z)^{-1}.
\end{split}\end{equation}
Then, by appealing to the partial fraction decomposition method, we can write, for $n\geq1 $,
\begin{equation}\label{Fa.4}\begin{split}
e^{(1)}_n(z)=[r^{(1)}(z)]^n q^{(1)}(z),
\quad
r^{(1)}(z)=\frac{1+z(\frac{1}{2}-\frac{\theta}{\alpha})}{1-z(\frac{1}{2}+\frac{\theta}{\alpha})},
\\
q^{(1)}(z)=\frac{1}{[1-z(\frac{1}{2}+\frac{\theta}{\alpha})][1+z(\frac{1}{2}-\frac{\theta}{\alpha})]},
\end{split}\end{equation}
and
\begin{equation}\label{Fa.4.1}\begin{split}
e^{(2)}_n(z)=[r^{(2)}(z)]^n q^{(2)}(z),
\quad
r^{(2)}(z)=\frac{1}{1-z},
\quad
q^{(2)}(z)=\frac{1}{1-z}.
\end{split}\end{equation}
Combining (\ref{Fa.1}) with (\ref{Fa.3}), we obtain
\begin{equation}\label{Fa.5}\begin{split}
\omega_n=\frac{\tau^{\alpha+1}}{2\pi{\rm i}}\int_{\mathcal{C}}e^{(i)}_n(\tau\lambda)F^{(i)}(\lambda)\mathrm{d}\lambda,\quad i=1,2.
\end{split}\end{equation}
To efficiently estimate $\omega_n$ by (\ref{Fa.5}), we divide the time interval into fast-increasing subintervals denoted by $I_{\ell}=[B^{\ell-1}\tau,(2B^\ell-2)\tau]$ where $B>1$ is an integer.
Let $T_\ell$ be the right end point of $I_\ell$.
If $n\tau \in I_{\ell}$ for some $\ell \geq 1$, we employ the trapezoidal rule to (\ref{Fa.5}) and obtain
\begin{equation}\label{Fa.6}\begin{split}
\omega_n \approx
\tau^{\alpha+1}\sum_{j=-K_{\mathcal{C}}}^{K_{\mathcal{C}}}w^{(\ell)}_j e^{(i)}_n(\tau\lambda_j^{(\ell)})F^{(i)}(\lambda_j^{(\ell)}),\quad i=1,2,
\end{split}\end{equation}
where
\begin{equation}\label{Fa.7}\begin{split}
w^{(\ell)}_j=\frac{\partial_\vartheta \lambda(\vartheta,\frac{K_{\mathcal{C}}}{T_\ell})}{2{\rm i}(K_{\mathcal{C}}+1)},
\quad
\lambda_j^{(\ell)}=\lambda(\vartheta_j,\frac{K_{\mathcal{C}}}{T_\ell}),\quad
\vartheta_j=\frac{j\pi}{K_{\mathcal{C}}+1}.
\end{split}\end{equation}
In our numerical tests, $B=5$ is taken, since with $K_{\mathcal{C}}=30$, the right-hand side of (\ref{Fa.6}) approximates $\omega_n$ quite well for $n\geq n_0$, e.g., $n_0=50$.
We remark that the algorithms ($\textbf{I}$) and $(\textbf{II})$ differ from one another in the functions $e_n^{(i)}(z)$ and $F^{(i)}(\lambda)$.
The contour integrals in (\ref{Fa.1}) require that the singularities of $e_n^{(i)}(z)$ lie to the right of the contour, while the singularities of $F^{(i)}(\lambda)$ are to the opposite side.
In Fig.\ref{C6} (a), we plot several Talbot contours $\mathcal{C}$.
Obviously, the singularities of $F^{(1)}(\lambda)$ and $e^{(2)}(z)$ meet the requirement.
\par
For $e^{(1)}(z)$ of (\textbf{I}), the singularity is
\begin{equation}\label{Fa.8}\begin{split}
z=\frac{1}{\frac{1}{2}+\frac{\theta}{\alpha}}.
\end{split}\end{equation}
If $\alpha \to 0$, we should take $\theta$ such that $z$ is bounded away from $0$ to meet the requirement.
Considering the first few weights cannot be accurately calculated by (\ref{Fa.6})\cite{Schadle,Zeng1}, we find numerically when $\frac{\theta}{\alpha} \in (\frac{1}{7},2)$ under the setting $B=5, K_{\mathcal{C}}=30$,
the absolute error of approximation (\ref{Fa.6}) is around $10^{-13}$ for $n\geq n_0=50$, as is shown in Fig.\ref{C7}.
\par
For $F^{(2)}(\lambda)$ of (\textbf{II}), the singularity is
\begin{equation}\label{Fa.9}\begin{split}
\lambda=
\begin{cases}
  \displaystyle\frac{1}{\tau}\frac{1}{\frac{1}{2}-\frac{\theta}{\alpha}}, & \theta\neq \frac{\alpha}{2}, \\
  0, & \theta = \frac{\alpha}{2}.
\end{cases}
\end{split}\end{equation}
We thus choose $\theta$ satisfying $\theta \in [\frac{\alpha}{2},\frac{1}{2}]$ such that $z=\lambda\tau\leq 0$.
Fig.\ref{C8} illustrates the impacts of $\theta$ on the absolute errors of weights.
We note for $\alpha=0.99$, when $\theta=0.2$ or $0.4$, we can still get a high accuracy although $\theta \leq \frac{\alpha}{2}$ for both cases.
\par
Now we are ready to present the fast algorithm based on the formula (\ref{Fa.6}).
\\
\\
\begin{tabular}{p{0.95\textwidth}}
\hline
\textbf{Algorithm.} Fast algorithm for problem (\ref{Se.14}).\\
\hline
1. Choose $B=5$, $K_{\mathcal{C}}=30$.
Input the parameters $\omega_j^{(\ell)}, \lambda_j^{(\ell)}$ and $\vartheta_j$ defined in (\ref{Fa.7}).
\\
2. For each $n\geq 1$, let $n=b_0>b_1>\cdots>b_{L-1}>b_L=0$, where $b_i$ can be calculated by a pseudocode in \cite{Schadle}.
\\
3. For some $\mathcal{M}>0$ (e.g., $\mathcal{M}=2$), if $L>\mathcal{M}$, re-express $D_\tau^\alpha W_h^{n-\theta}$ as
$$D_\tau^\alpha W_h^{n-\theta}=\sum_{\ell=0}^{\mathcal{M}}\widetilde{W}_n^{(\ell)}+\sum_{\ell=\mathcal{M}+1}^{L}\widetilde{W}_n^{(\ell)}$$
where
$$\widetilde{W}_n^{(\ell)}=
\begin{cases}
  \tau^{-\alpha}\omega_0 W_h^n, & \ell=0, \\
  \displaystyle\tau^{-\alpha}\sum_{k=b_\ell}^{b_{\ell-1}-1}\omega_{n-k}W_h^k, & \ell=1,2,\cdots,L.
\end{cases}
$$
4. Replace $\omega_n (n\geq 1)$ in $\widetilde{W}_n^{(\ell)}$ by the approximation formula (\ref{Fa.6}) to obtain
$$
\widetilde{W}_n^{(\ell)}\approx \sum_{j=-K_{\mathcal{C}}}^{K_{\mathcal{C}}}\omega_j^{(\ell)} \big[r^{(i)}(\tau\lambda_j^{(\ell)})\big]^{n-(b_{\ell-1}-1)}y_j^{(i)}F^{(i)}(\lambda_j^{(\ell)}),\quad i=1,2,
$$
where
$
y_j^{(i)}=y_j^{(i)}(b_\ell,b_{\ell-1},\lambda_j^{(\ell)})=\tau\sum_{k=b_\ell}^{b_{\ell-1}-1}e_{(b_{\ell-1}-1)-k}^{(i)}(\tau\lambda_j^{(\ell)})W_h^k
$ satisfies some recursive formula.
\\
5. For those $n$ leading to $L \leq \mathcal{M}$, standard algorithm is adopt to get $W_h^n$.
If $L > \mathcal{M}$, problem (\ref{Se.14}), according to Steps $3$-$4$, can be reformulated as
$$
\widetilde{W}_n^{(0)}-(1-\theta)\Delta_h W_h^n=\theta \Delta_hW_h^{n-1}
-\sum_{\ell=1}^{\mathcal{M}}\widetilde{W}_n^{(\ell)}
-\sum_{\ell=\mathcal{M}+1}^{L}\widetilde{W}_n^{(\ell)}
+f_h^{n-\theta}+\Delta_h v_h.
$$
\\
\hline
\end{tabular}
\begin{figure}[htbp]
\centering
\subfigure[]{
\begin{minipage}[t]{0.5\linewidth}
\centering
\includegraphics[width=0.9\textwidth]{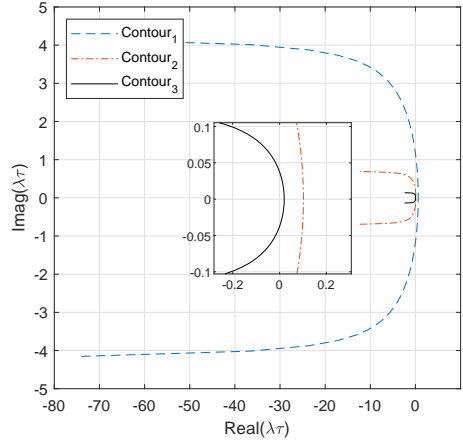}
\end{minipage}%
}%
\centering
\caption{(a) Talbot contour$_\ell$: $\lambda(\vartheta,\frac{K_{\mathcal{C}}}{T_\ell})$ with $K_{\mathcal{C}}=30$, $\ell=1,2,3$ and $B=5$.}\label{C6}
\end{figure}
\begin{figure}[htbp]
\centering
\subfigure[]{
\begin{minipage}[t]{0.45\linewidth}
\centering
\includegraphics[width=1\textwidth]{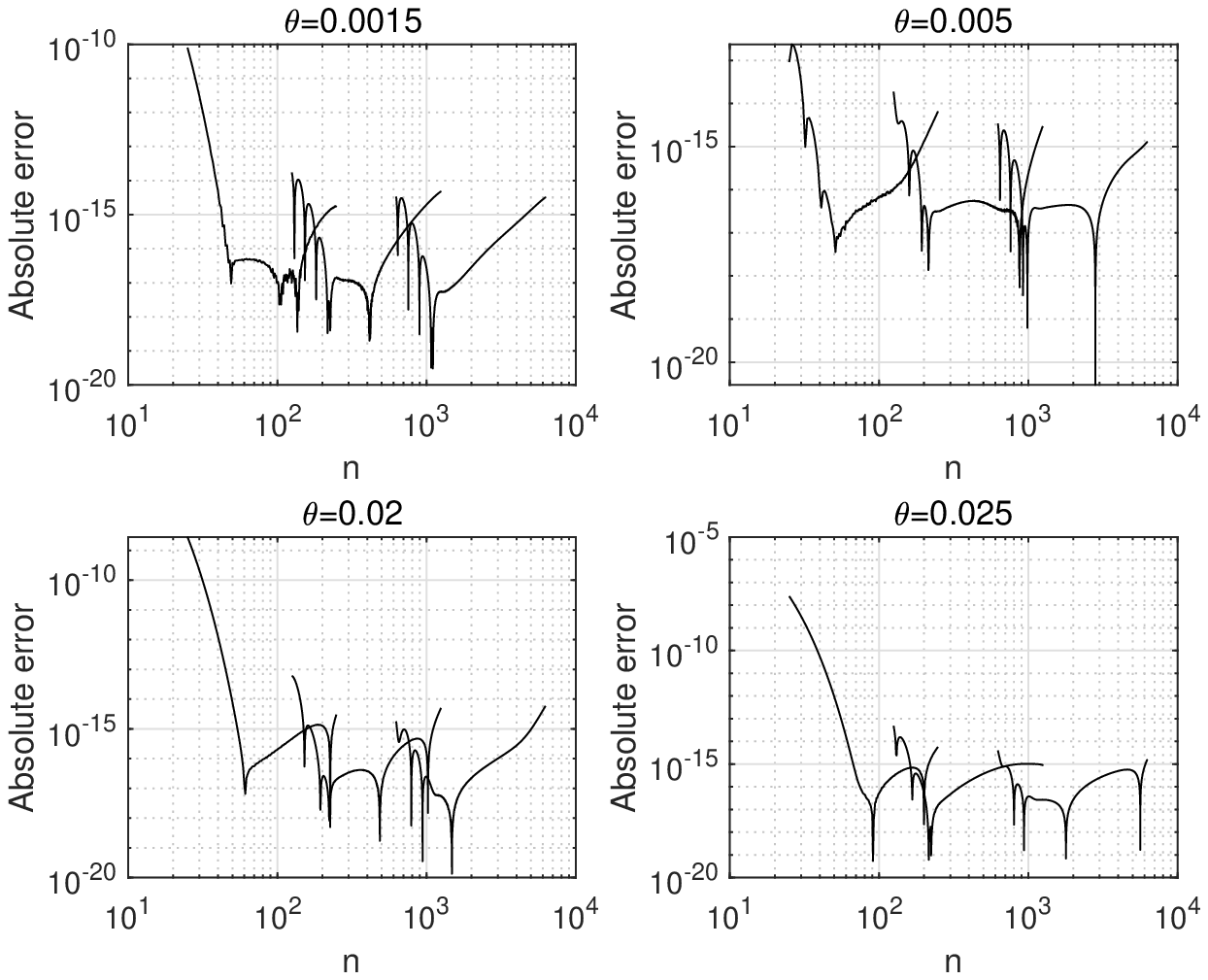}
\end{minipage}%
}%
\subfigure[]{
\begin{minipage}[t]{0.45\linewidth}
\centering
\includegraphics[width=1\textwidth]{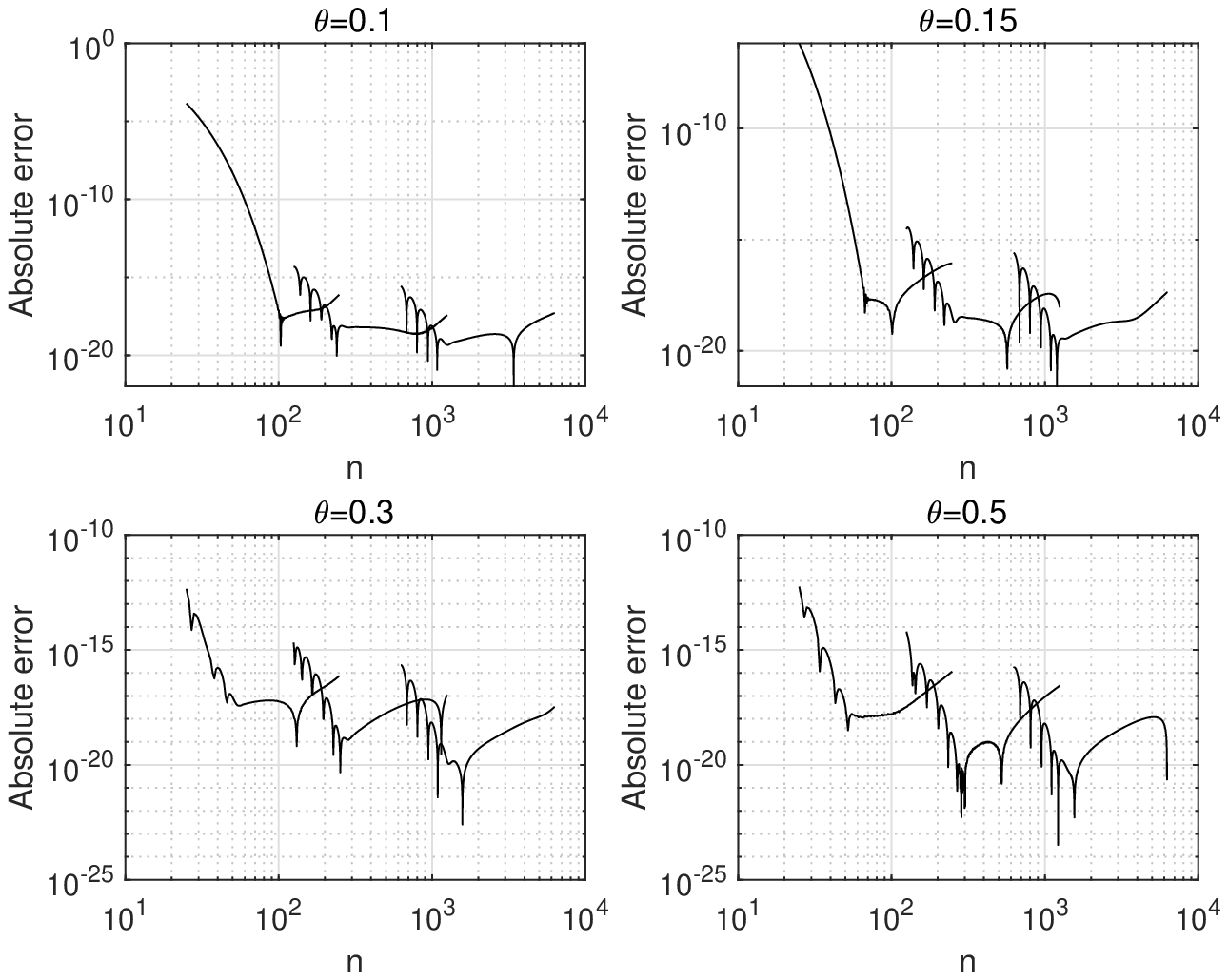}
\end{minipage}%
}%
\centering
\caption{(a) Absolute errors for $\omega_n$ of (\textbf{I}), $\alpha=0.01$. (b) Absolute errors for $\omega_n$ of (\textbf{I}), $\alpha=0.99$.}\label{C7}
\end{figure}
\begin{figure}[htbp]
\centering
\subfigure[]{
\begin{minipage}[t]{0.45\linewidth}
\centering
\includegraphics[width=1\textwidth]{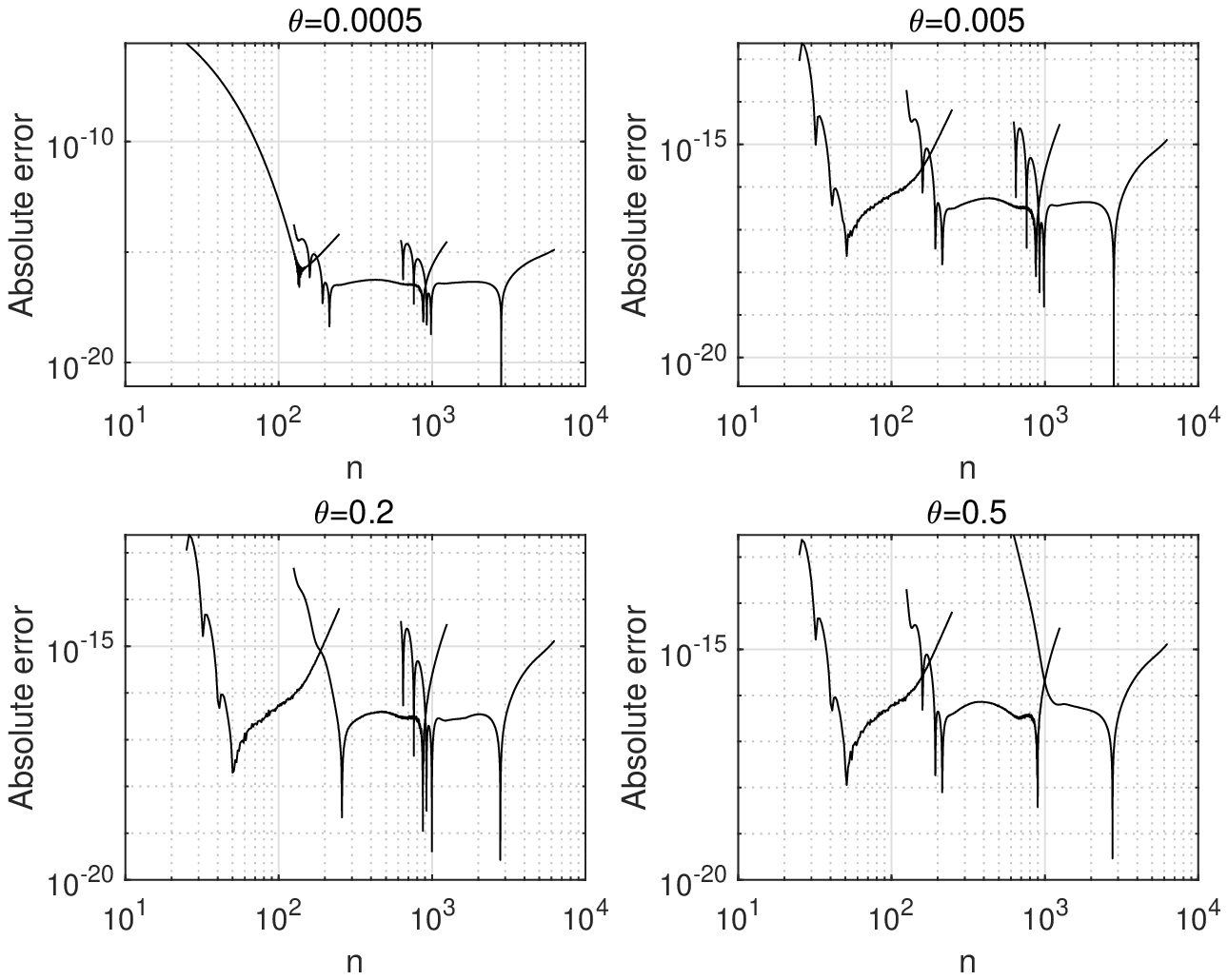}
\end{minipage}%
}%
\subfigure[]{
\begin{minipage}[t]{0.45\linewidth}
\centering
\includegraphics[width=1\textwidth]{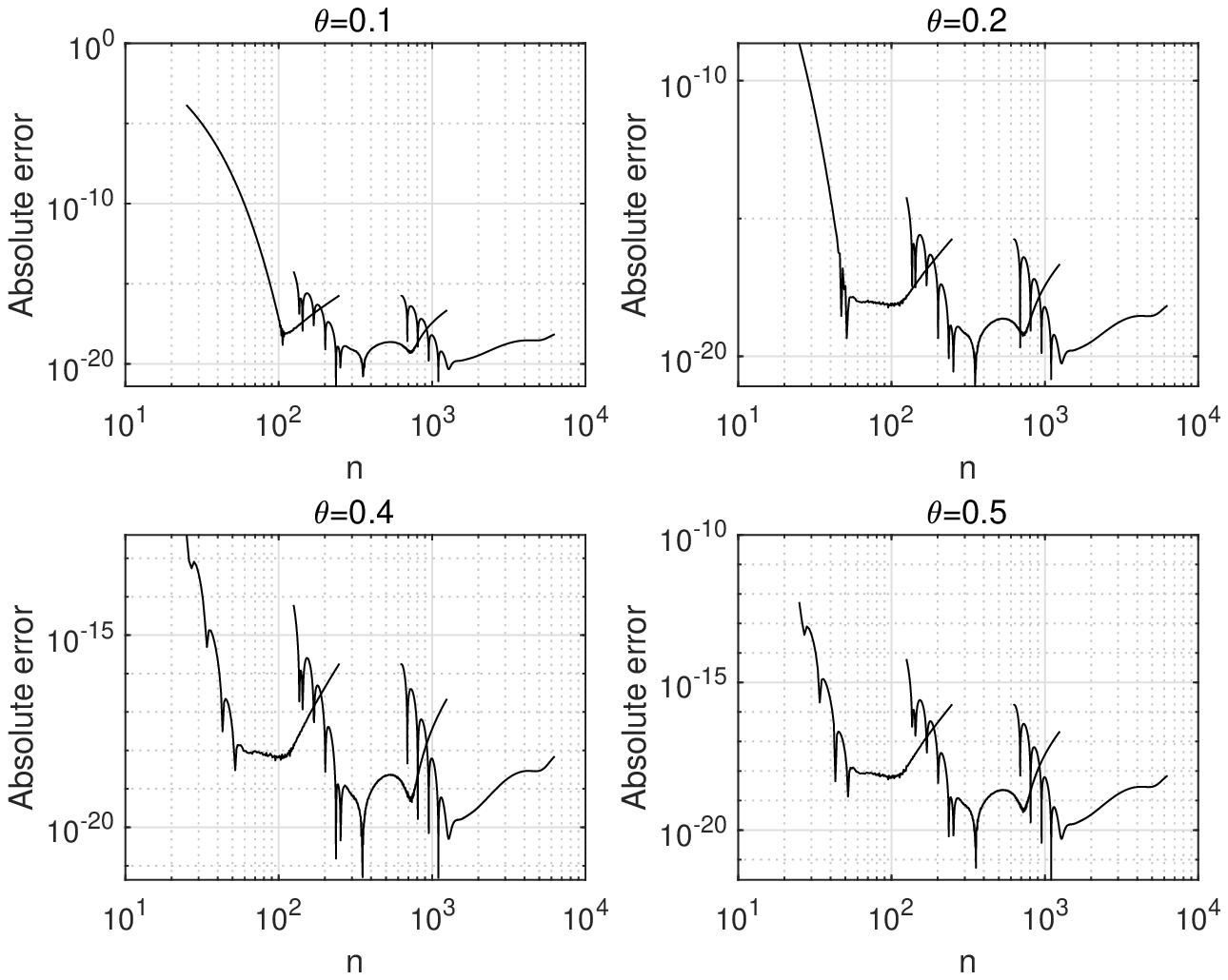}
\end{minipage}%
}%
\centering
\caption{(a) Absolute errors for $\omega_n$ of (\textbf{II}), $\alpha=0.01$. (b) Absolute errors for $\omega_n$ of (\textbf{II}), $\alpha=0.99$.}\label{C8}
\end{figure}
\section{Numerical tests}\label{sec.num}
\textit{Example 1.}
\par
\begin{table}[]
\centering
\caption{Temporal convergence rates of \textit{Example 1} at $t=0.5$ without initial corrections}\label{tab1}
{\small
{\renewcommand{\arraystretch}{1.2}
\begin{tabular}{ccrrrrrrrrr}
\toprule
$\alpha$             & $\theta$ & \multicolumn{1}{c}{$\tau=2^{-5}$} & \multicolumn{1}{c}{$\tau=2^{-6}$} & \multicolumn{1}{c}{$\tau=2^{-7}$} & \multicolumn{1}{c}{$\tau=2^{-8}$} & \multicolumn{1}{c}{$\tau=2^{-9}$} & \multicolumn{1}{l}{Rate} & \multicolumn{1}{l}{Rate} & \multicolumn{1}{l}{Rate} & \multicolumn{1}{l}{Rate} \\ \hline
\multirow{5}{*}{0.1} & 0.1      & 7.93E-04                          & 3.95E-04                          & 1.97E-04                          & 9.84E-05                          & 4.92E-05                          & 1.01                     & 1.00                     & 1.00                     & 1.00                     \\
                     & 0.2      & 3.88E-04                          & 2.42E-04                          & 1.34E-04                          & 7.02E-05                          & 3.60E-05                          & 0.68                     & 0.85                     & 0.93                     & 0.97                     \\
                     & 0.3      & 2.04E-04                          & 4.11E-05                          & 5.79E-05                          & 3.89E-05                          & 2.20E-05                          & 2.31                     & -0.49                    & 0.58                     & 0.82                     \\
                     & 0.4      & 1.01E-03                          & 2.05E-04                          & 3.00E-05                          & 4.32E-06                          & 7.16E-06                          & 2.30                     & 2.77                     & 2.80                     & -0.73                    \\
                     & 0.5      & 2.04E-03                          & 4.97E-04                          & 1.30E-04                          & 3.34E-05                          & 8.46E-06                          & 2.04                     & 1.94                     & 1.96                     & 1.98                     \\ \hline
\multirow{5}{*}{0.5} & 0.1      & 4.50E-03                          & 2.20E-03                          & 1.09E-03                          & 5.41E-04                          & 2.70E-04                          & 1.03                     & 1.02                     & 1.01                     & 1.00                     \\
                     & 0.2      & 3.39E-03                          & 1.65E-03                          & 8.17E-04                          & 4.06E-04                          & 2.02E-04                          & 1.03                     & 1.02                     & 1.01                     & 1.00                     \\
                     & 0.3      & 2.22E-03                          & 1.09E-03                          & 5.42E-04                          & 2.70E-04                          & 1.35E-04                          & 1.02                     & 1.01                     & 1.01                     & 1.00                     \\
                     & 0.4      & 9.82E-04                          & 5.14E-04                          & 2.63E-04                          & 1.33E-04                          & 6.69E-05                          & 0.93                     & 0.97                     & 0.98                     & 0.99                     \\
                     & 0.5      & 3.17E-04                          & 8.03E-05                          & 2.02E-05                          & 5.08E-06                          & 1.28E-06                          & 1.98                     & 1.99                     & 1.99                     & 1.99                     \\ \hline
\multirow{5}{*}{0.9} & 0.1      & 9.02E-03                          & 4.40E-03                          & 2.17E-03                          & 1.08E-03                          & 5.38E-04                          & 1.04                     & 1.02                     & 1.01                     & 1.00                     \\
                     & 0.2      & 6.88E-03                          & 3.33E-03                          & 1.64E-03                          & 8.11E-04                          & 4.04E-04                          & 1.05                     & 1.02                     & 1.01                     & 1.01                     \\
                     & 0.3      & 4.68E-03                          & 2.24E-03                          & 1.10E-03                          & 5.42E-04                          & 2.70E-04                          & 1.06                     & 1.03                     & 1.02                     & 1.01                     \\
                     & 0.4      & 2.42E-03                          & 1.14E-03                          & 5.53E-04                          & 2.72E-04                          & 1.35E-04                          & 1.08                     & 1.04                     & 1.02                     & 1.01                     \\
                     & 0.5      & 9.24E-05                          & 2.30E-05                          & 5.74E-06                          & 1.43E-06                          & 3.50E-07                          & 2.01                     & 2.00                     & 2.01                     & 2.03                     \\ \bottomrule
\end{tabular}}}
\end{table}

\begin{table}[]
\centering
\caption{Temporal convergence rates of \textit{Example 1} at $t=0.5$ with initial corrections}\label{tab2}
{\small
{\renewcommand{\arraystretch}{1.2}
\begin{tabular}{ccrrrrrrrrr}
\toprule
$\alpha$             & $\theta$ & \multicolumn{1}{c}{$\tau=2^{-5}$} & \multicolumn{1}{c}{$\tau=2^{-6}$} & \multicolumn{1}{c}{$\tau=2^{-7}$} & \multicolumn{1}{c}{$\tau=2^{-8}$} & \multicolumn{1}{c}{$\tau=2^{-9}$} & \multicolumn{1}{l}{Rate} & \multicolumn{1}{l}{Rate} & \multicolumn{1}{l}{Rate} & \multicolumn{1}{l}{Rate} \\ \hline
\multirow{5}{*}{0.1} & 0.1      & 3.99E-05                          & 9.81E-06                          & 2.44E-06                          & 6.13E-07                          & 1.58E-07                          & 2.03                     & 2.01                     & 1.99                     & 1.95                     \\
                     & 0.2      & 2.32E-04                          & 6.09E-05                          & 1.56E-05                          & 3.96E-06                          & 1.00E-06                          & 1.93                     & 1.96                     & 1.98                     & 1.98                     \\
                     & 0.3      & 6.13E-04                          & 1.60E-04                          & 4.15E-05                          & 1.06E-05                          & 2.67E-06                          & 1.94                     & 1.95                     & 1.97                     & 1.98                     \\
                     & 0.4      & 1.21E-03                          & 3.05E-04                          & 7.96E-05                          & 2.04E-05                          & 5.16E-06                          & 1.99                     & 1.94                     & 1.97                     & 1.98                     \\
                     & 0.5      & 2.04E-03                          & 4.97E-04                          & 1.30E-04                          & 3.34E-05                          & 8.46E-06                          & 2.04                     & 1.94                     & 1.96                     & 1.98                     \\ \hline
\multirow{5}{*}{0.5} & 0.1      & 4.82E-05                          & 1.01E-05                          & 2.31E-06                          & 5.54E-07                          & 1.41E-07                          & 2.25                     & 2.13                     & 2.06                     & 1.97                     \\
                     & 0.2      & 1.44E-06                          & 1.69E-06                          & 5.79E-07                          & 1.59E-07                          & 3.62E-08                          & -0.23                    & 1.54                     & 1.87                     & 2.13                     \\
                     & 0.3      & 2.70E-05                          & 6.10E-06                          & 1.45E-06                          & 3.58E-07                          & 9.40E-08                          & 2.14                     & 2.07                     & 2.02                     & 1.93                     \\
                     & 0.4      & 1.33E-04                          & 3.34E-05                          & 8.39E-06                          & 2.10E-06                          & 5.33E-07                          & 1.99                     & 2.00                     & 1.99                     & 1.98                     \\
                     & 0.5      & 3.17E-04                          & 8.03E-05                          & 2.02E-05                          & 5.08E-06                          & 1.28E-06                          & 1.98                     & 1.99                     & 1.99                     & 1.99                     \\ \hline
\multirow{5}{*}{0.9} & 0.1      & 1.15E-04                          & 3.01E-05                          & 7.68E-06                          & 1.93E-06                          & 4.79E-07                          & 1.94                     & 1.97                     & 1.99                     & 2.01                     \\
                     & 0.2      & 2.31E-04                          & 5.86E-05                          & 1.48E-05                          & 3.70E-06                          & 9.20E-07                          & 1.98                     & 1.99                     & 2.00                     & 2.01                     \\
                     & 0.3      & 2.65E-04                          & 6.69E-05                          & 1.68E-05                          & 4.20E-06                          & 1.04E-06                          & 1.99                     & 1.99                     & 2.00                     & 2.01                     \\
                     & 0.4      & 2.19E-04                          & 5.50E-05                          & 1.38E-05                          & 3.44E-06                          & 8.55E-07                          & 1.99                     & 2.00                     & 2.00                     & 2.01                     \\
                     & 0.5      & 9.24E-05                          & 2.30E-05                          & 5.74E-06                          & 1.43E-06                          & 3.50E-07                          & 2.01                     & 2.00                     & 2.01                     & 2.03                     \\ \bottomrule
\end{tabular}}}
\end{table}
In this example, we show the efficiency of the correction technique for the scheme.
Let $\Omega=(0,\pi)$ and $T=1$.
Take $v(x)=\sin x$ and $f=\big(6t^{3-\alpha}/\Gamma(4-\alpha)+t^3\big)\sin x$, then the exact solution is $u(x,t)=(E_\alpha(-t^\alpha)+t^3)\sin x$, where $E_\alpha(t)$ denotes the Mittag-Leffler function defined by $E_\alpha(t)=\sum_{j=0}^{\infty}\frac{t^j}{\Gamma(j\alpha+1)}$.
In this test, we take a fine space mesh $h=10^{-4}$.
Errors in the $L^2(\Omega)$ norm at fixed time $t=0.5$  are reported in Table \ref{tab1} and Table \ref{tab2} for different $\alpha$ and $\theta$.
One can find without initial corrections, the scheme has only first-order convergence rate for $\theta \in (0,\frac{1}{2})$, due to the weak singularity of the solution.
However, by adding initial corrections, we can recover the optimal convergence rate as demonstrated in Table \ref{tab2}.
\par
An interesting phenomenon occurs if we take $\theta=\frac{1}{2}$, as is shown in Table \ref{tab1}.
In this case, we can get second-order convergence rate without initial corrections.
As is pointed out in Sect.\ref{sec.sec}, the special case $\theta=\frac{1}{2}$ can not be covered theoretically in this study, and indeed we will in Example 4 further show that if we approximate the fractional derivative $D_t^\alpha u(t_{n-\frac{1}{2}})$ by a quite different approach, like using the Crank-Nicolson+fractional BDF2, we can not obtain the optimal convergence rate, or, even the first-order accuracy, which complicates the underlying analysis for this special case $\theta=\frac{1}{2}$.
\\
\\
\textit{Example 2.}
\par
We validate the sharpness of the results of Theorem \ref{thm.4} in this example.
Let $T=1$, and
\par
(i) $f=0$, $v=\sin x \in D(\Delta)$, $\Omega=(0,\pi)$, with the exact solution $u(x,t)=E_\alpha(-t^\alpha)\sin x$;
\par
(ii) $f=0$, $v=\chi_{(0,1/2)}$, $\Omega=(0,1)$;
\par
For the case (i), the results in Table \ref{tab3} indicate that our scheme with initial corrections can obtain the second-order convergence rate.
Note also that if $\theta=\frac{1}{2}$, the optimal convergence rate is recovered.
To further confirm the prefactor $t_n^{2-\alpha}$ in the error estimates, we depict in Fig.\ref{C1} the behaviors of $L^2$ errors as $t \to 0$ for $\theta=0.1$ and $\theta=0.45$, respectively.
One can find the $L^2$ errors deteriorate as $t \to 0$ with different speeds depending on $\alpha$.
By comparing with the non-solid lines which are obtained using the formula $O(t_n^{\alpha-2}\tau^2)$, we see that our error estimate is sharp.
\par
The example (ii) is to check the case $v\in L^2(\Omega)$ and $v$ is nonsmooth.
Since the exact solution is unknown in advance, we take the numerical solution on a fine mesh $h=10^{-4}, \tau=2^{-12}$ as a reference.
Empirical results in Table \ref{tab4} show that the scheme (\ref{Se.14}) can result in the optimal convergence rate.
However, the convergence rate for $\theta=\frac{1}{2}$ is much lower than that of the smooth case (i).
In Fig.\ref{C2}, we examine the prefactor $t_n^{-2}$ in error estimates for this case, and find that $\alpha$ has little effect on the behavior of error as $t\to 0$, which is in line with the theoretical analysis.
\begin{table}[]
\centering
\caption{Temporal convergence rates of \textit{Example 2} (i) at $t=0.5$ for scheme (\ref{Se.14})}\label{tab3}
{\small
{\renewcommand{\arraystretch}{1.2}
\begin{tabular}{ccrrrrrrrrr}
\toprule
$\alpha$             & $\theta$ & \multicolumn{1}{c}{$\tau=2^{-5}$} & \multicolumn{1}{c}{$\tau=2^{-6}$} & \multicolumn{1}{c}{$\tau=2^{-7}$} & \multicolumn{1}{c}{$\tau=2^{-8}$} & \multicolumn{1}{c}{$\tau=2^{-9}$} & \multicolumn{1}{l}{Rate} & \multicolumn{1}{l}{Rate} & \multicolumn{1}{l}{Rate} & \multicolumn{1}{l}{Rate} \\ \hline
\multirow{3}{*}{0.2} & 0.20     & 8.26E-05                          & 2.02E-05                          & 4.98E-06                          & 1.24E-06                          & 3.16E-07                          & 2.03                     & 2.02                     & 2.00                     & 1.98                     \\
                     & 0.35     & 6.96E-05                          & 1.67E-05                          & 4.13E-06                          & 1.03E-06                          & 2.62E-07                          & 2.06                     & 2.02                     & 2.00                     & 1.98                     \\
                     & 0.50     & 7.81E-05                          & 1.49E-05                          & 3.59E-06                          & 8.90E-07                          & 2.27E-07                          & 2.39                     & 2.05                     & 2.01                     & 1.97                     \\ \hline
\multirow{3}{*}{0.5} & 0.10     & 2.27E-04                          & 5.50E-05                          & 1.36E-05                          & 3.37E-06                          & 8.46E-07                          & 2.04                     & 2.02                     & 2.01                     & 2.00                     \\
                     & 0.30     & 1.60E-04                          & 3.95E-05                          & 9.79E-06                          & 2.44E-06                          & 6.15E-07                          & 2.02                     & 2.01                     & 2.00                     & 1.99                     \\
                     & 0.50     & 1.35E-04                          & 3.31E-05                          & 8.18E-06                          & 2.04E-06                          & 5.14E-07                          & 2.03                     & 2.01                     & 2.00                     & 1.99                     \\ \hline
\multirow{3}{*}{0.9} & 0.05     & 2.43E-04                          & 5.93E-05                          & 1.47E-05                          & 3.66E-06                          & 9.19E-07                          & 2.03                     & 2.01                     & 2.00                     & 1.99                     \\
                     & 0.45     & 7.41E-05                          & 1.85E-05                          & 4.63E-06                          & 1.16E-06                          & 2.97E-07                          & 2.00                     & 2.00                     & 1.99                     & 1.97                     \\
                     & 0.50     & 7.94E-05                          & 1.98E-05                          & 4.93E-06                          & 1.24E-06                          & 3.16E-07                          & 2.01                     & 2.00                     & 1.99                     & 1.97                     \\ \bottomrule
\end{tabular}}}
\end{table}
\begin{table}[]
\centering
\caption{Temporal convergence rates of \textit{Example 2} (ii) at $t=0.5$  for scheme (\ref{Se.14})}\label{tab4}
{\small
{\renewcommand{\arraystretch}{1.2}
\begin{tabular}{cccccccrrrr}
\toprule
\multicolumn{1}{l}{$\alpha$} & $\theta$ & $\tau=2^{-5}$ & $\tau=2^{-6}$ & $\tau=2^{-7}$ & $\tau=2^{-8}$ & $\tau=2^{-9}$ & \multicolumn{1}{l}{Rate} & \multicolumn{1}{l}{Rate} & \multicolumn{1}{l}{Rate} & \multicolumn{1}{l}{Rate} \\ \hline
\multirow{3}{*}{0.2}         & 0.20      & 1.17E-05      & 2.86E-06      & 7.08E-07      & 1.75E-07      & 4.27E-08      & 2.03                     & 2.02                     & 2.01                     & 2.04                     \\
                             & 0.35     & 2.15E-05      & 7.42E-07      & 2.05E-07      & 5.31E-08      & 1.36E-08      & 4.86                     & 1.85                     & 1.95                     & 1.97                     \\
                             & 0.50      & 2.70E-01      & 2.09E-01      & 1.67E-01      & 1.25E-01      & 8.87E-02      & 0.37                     & 0.33                     & 0.42                     & 0.49                     \\ \hline
\multirow{3}{*}{0.5}         & 0.10      & 4.90E-05      & 1.19E-05      & 2.91E-06      & 7.20E-07      & 1.77E-07      & 2.05                     & 2.03                     & 2.02                     & 2.02                     \\
                             & 0.30      & 3.45E-05      & 8.49E-06      & 2.11E-06      & 5.22E-07      & 1.29E-07      & 2.02                     & 2.01                     & 2.01                     & 2.02                     \\
                             & 0.50      & 2.04E-01      & 1.56E-01      & 1.17E-01      & 8.51E-02      & 5.91E-02      & 0.39                     & 0.42                     & 0.45                     & 0.53                     \\ \hline
\multirow{3}{*}{0.9}         & 0.05     & 1.90E-04      & 3.54E-05      & 8.63E-06      & 2.13E-06      & 5.22E-07      & 2.43                     & 2.04                     & 2.02                     & 2.03                     \\
                             & 0.45     & 5.71E-03      & 1.67E-04      & 5.28E-06      & 1.31E-06      & 3.24E-07      & 5.10                     & 4.98                     & 2.01                     & 2.02                     \\
                             & 0.50      & 1.29E-01      & 9.23E-02      & 6.53E-02      & 4.53E-02      & 3.03E-02      & 0.49                     & 0.50                     & 0.53                     & 0.58                     \\ \bottomrule
\end{tabular}}}
\end{table}
\begin{figure}[htbp]
\centering

\subfigure[]{
\begin{minipage}[t]{0.5\linewidth}
\centering
\includegraphics[width=1\textwidth]{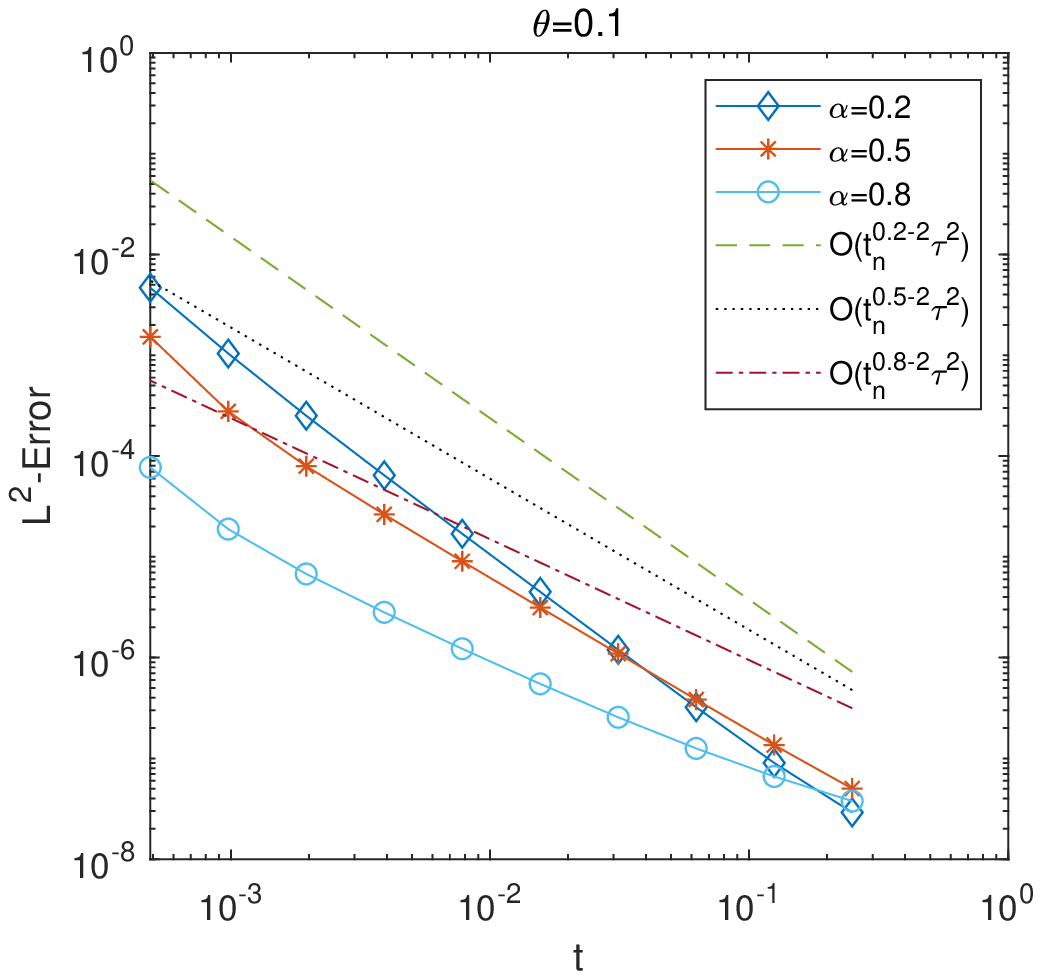}
\end{minipage}%
}%
\subfigure[]{
\begin{minipage}[t]{0.5\linewidth}
\centering
\includegraphics[width=1\textwidth]{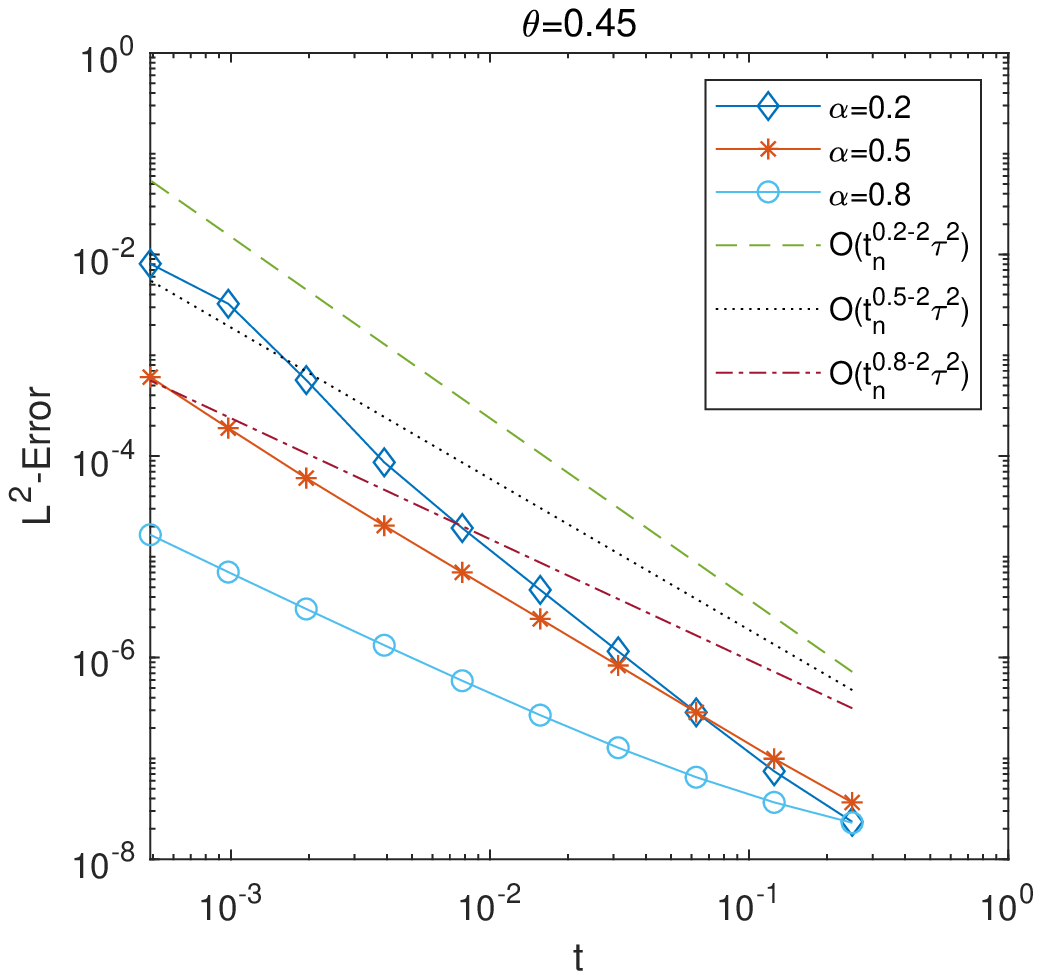}
\end{minipage}%
}%
\centering
\caption{$L^2$ errors of Example 2 (i) for time $t\to 0$ under the condition $N=2^{-12}, h=\pi\times 10^{-4}$. (a) $\theta=0.10$. (b) $\theta=0.45$.}\label{C1}
\end{figure}
\begin{figure}[htbp]
\centering

\subfigure[]{
\begin{minipage}[t]{0.5\linewidth}
\centering
\includegraphics[width=1\textwidth]{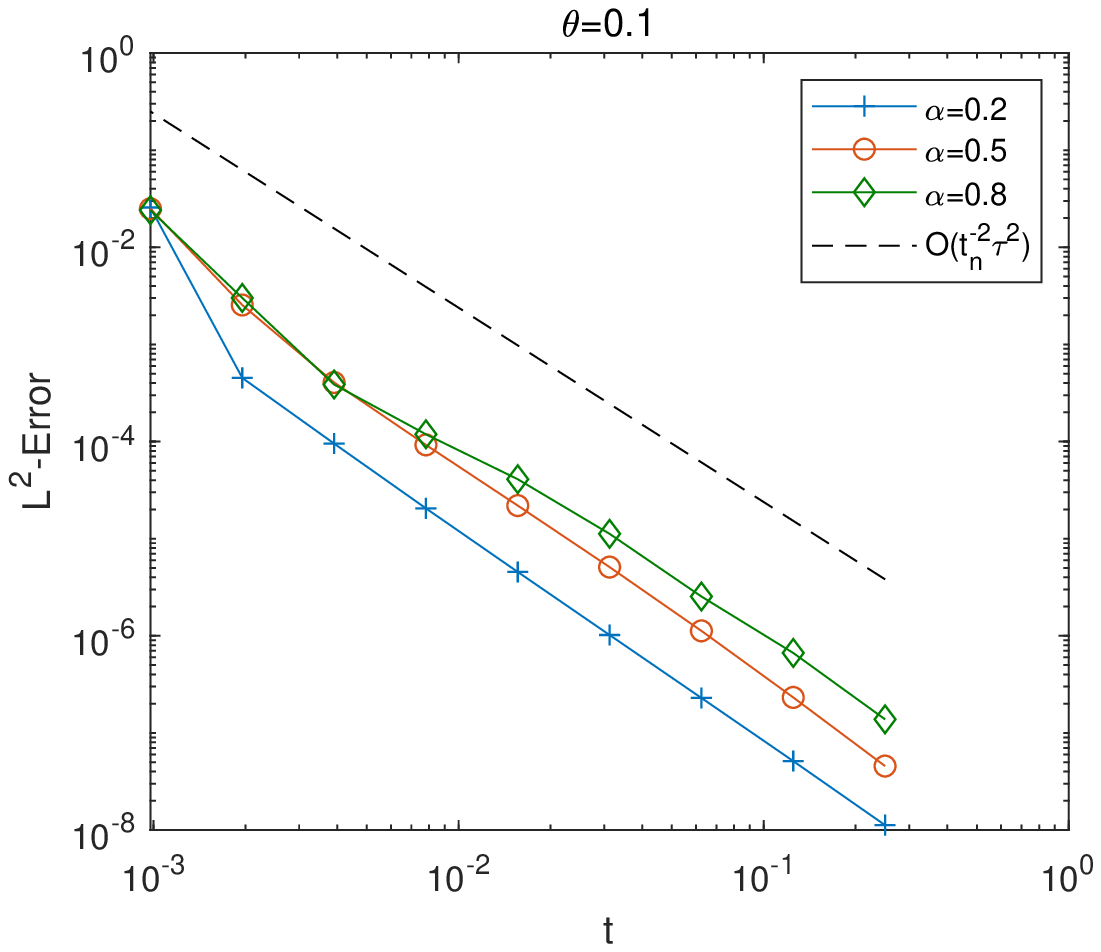}
\end{minipage}%
}%
\subfigure[]{
\begin{minipage}[t]{0.5\linewidth}
\centering
\includegraphics[width=1\textwidth]{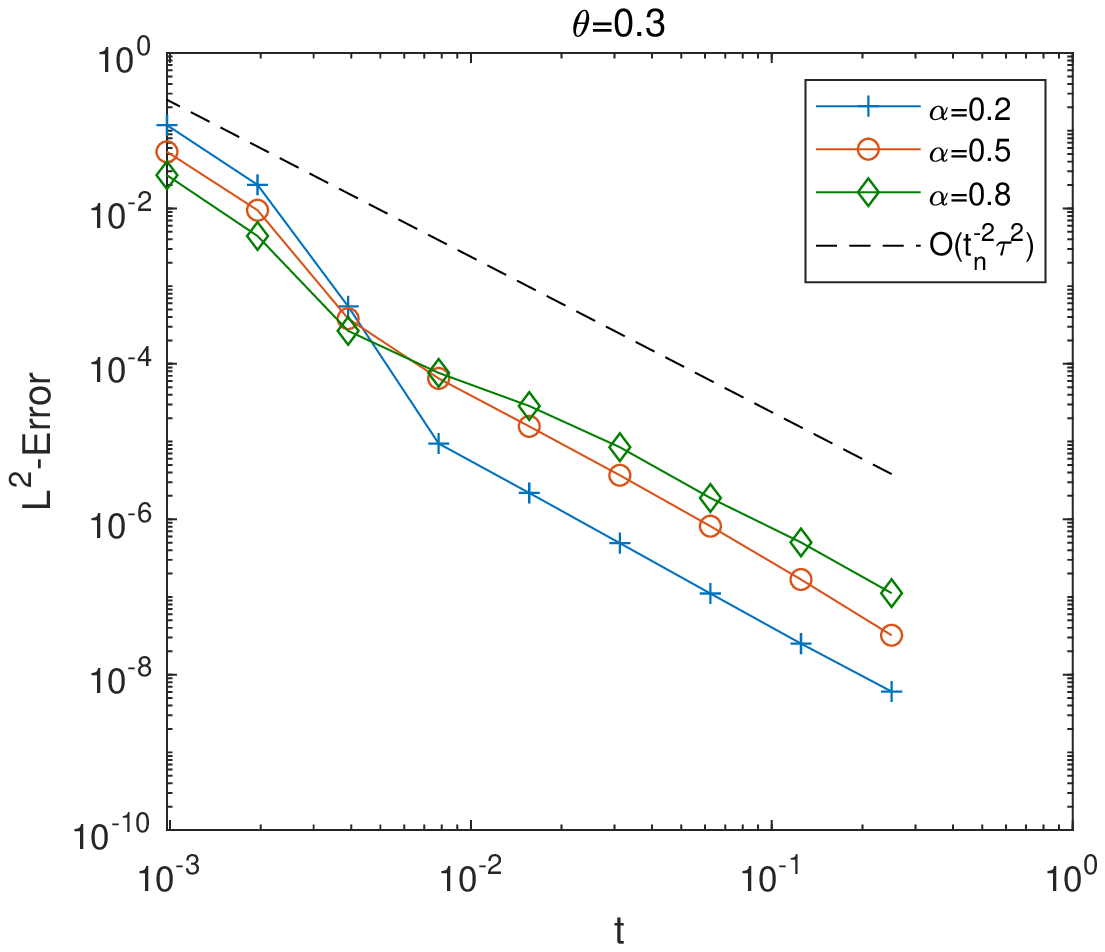}
\end{minipage}%
}%
\centering
\caption{$L^2$ errors of Example 2 (ii) for time $t\to 0$ under the condition $N=2^{-11}, h=10^{-4}$. (a) $\theta=0.1$. (b) $\theta=0.3$.}\label{C2}
\end{figure}
\\
\\
\textit{Example 3.}
\par
This example devotes to showing the unstability of the scheme (\ref{Se.14}) if $\theta>\frac{1}{2}$, and checking the robustness of the scheme when (i) $\alpha \to 1^-$, or $\alpha \to 0$, (ii) $\theta \to 0$.
We take $f=0, v=\sin x, x\in \Omega=(0,\pi)$ with the exact solution $u(x,t)=E_\alpha(-t^\alpha)\sin x$.
\par
The Fig.\ref{C3} (a) corresponds to the experiment by taking $\alpha=0.8$, $h=\pi\times10^{-2}$, $\tau=5\times2^{-9}$ and $T=10$, with different $\theta=0.509$, or $0.499$.
The figure is the numerical solution of $u(\pi/2,t)$ for $t \in [0,T]$.
One can see that for $\theta>\frac{1}{2}$, the scheme (\ref{Se.14}) may become unstable.
In Fig.\ref{C3} (b), we take $h=2\pi\times10^{-4}$, $\tau=2^{-10}$, $T=1$, and can observe the $L^2$ error varies smoothly for $\alpha$ bounded away from $0$ (in this test, for example, $\alpha \geq 0.02$), which confirms our theoretical analysis, as constants in Lemma \ref{lem.7} blow up if $\alpha \to 0$.
To check the case $\theta \to 0$, we take $\theta=0.1, 0.01, 0.001$ and $0$ for different $\alpha=0.1, 0.5$ and $0.9$ in Table \ref{tab5}, and find that our scheme is robust for quite small $\theta$ and the optimal convergence rate can also be obtained for $\theta=0$ if $\alpha$ is bounded away from $0$ such as $\alpha=0.5$ or $0.9$.
\begin{table}[]
\centering
\caption{Temporal convergence rates of \textit{Example 3} (ii) at $t=0.5$  for scheme (\ref{Se.14})}\label{tab5}
{\small
{\renewcommand{\arraystretch}{1.2}
\begin{tabular}{ccrrrrrrrrr}
\toprule
$\alpha$             & $\theta$ & \multicolumn{1}{c}{$\tau=2^{-5}$} & \multicolumn{1}{c}{$\tau=2^{-6}$} & \multicolumn{1}{c}{$\tau=2^{-7}$} & \multicolumn{1}{c}{$\tau=2^{-8}$} & \multicolumn{1}{c}{$\tau=2^{-9}$} & \multicolumn{1}{l}{Rate} & \multicolumn{1}{l}{Rate} & \multicolumn{1}{l}{Rate} & \multicolumn{1}{l}{Rate} \\ \hline
\multirow{4}{*}{0.1} & 0.100    & 4.79E-05                          & 1.16E-05                          & 2.86E-06                          & 7.16E-07                          & 1.83E-07                          & 2.04                     & 2.02                     & 2.00                     & 1.97                     \\
                     & 0.010    & 5.80E-05                          & 1.31E-05                          & 3.22E-06                          & 8.03E-07                          & 2.06E-07                          & 2.15                     & 2.02                     & 2.00                     & 1.97                     \\
                     & 0.001    & 1.01E-03                          & 2.44E-04                          & 3.29E-05                          & 1.87E-06                          & 2.10E-07                          & 2.04                     & 2.89                     & 4.14                     & 3.15                     \\
                     & 0.000    & 1.81E-03                          & 8.19E-04                          & 3.76E-04                          & 1.73E-04                          & 7.93E-05                          & 1.14                     & 1.12                     & 1.12                     & 1.12                     \\ \hline
\multirow{4}{*}{0.5} & 0.100    & 2.27E-04                          & 5.50E-05                          & 1.36E-05                          & 3.37E-06                          & 8.46E-07                          & 2.04                     & 2.02                     & 2.01                     & 2.00                     \\
                     & 0.010    & 4.22E-04                          & 7.52E-05                          & 1.62E-05                          & 3.96E-06                          & 9.92E-07                          & 2.49                     & 2.22                     & 2.03                     & 2.00                     \\
                     & 0.001    & 7.37E-04                          & 1.64E-04                          & 3.49E-05                          & 6.81E-06                          & 1.26E-06                          & 2.16                     & 2.23                     & 2.36                     & 2.44                     \\
                     & 0.000    & 7.98E-04                          & 1.93E-04                          & 4.73E-05                          & 1.17E-05                          & 2.93E-06                          & 2.05                     & 2.03                     & 2.01                     & 2.00                     \\ \hline
\multirow{4}{*}{0.9} & 0.100    & 2.01E-04                          & 4.93E-05                          & 1.22E-05                          & 3.05E-06                          & 7.66E-07                          & 2.03                     & 2.01                     & 2.00                     & 1.99                     \\
                     & 0.010    & 2.84E-04                          & 6.87E-05                          & 1.69E-05                          & 4.21E-06                          & 1.06E-06                          & 2.05                     & 2.02                     & 2.01                     & 1.99                     \\
                     & 0.001    & 2.97E-04                          & 7.16E-05                          & 1.76E-05                          & 4.36E-06                          & 1.09E-06                          & 2.05                     & 2.03                     & 2.01                     & 2.00                     \\
                     & 0.000    & 2.98E-04                          & 7.20E-05                          & 1.77E-05                          & 4.38E-06                          & 1.10E-06                          & 2.05                     & 2.03                     & 2.01                     & 2.00                     \\ \bottomrule
\end{tabular}}}
\end{table}
\begin{figure}[htbp]
\centering

\subfigure[]{
\begin{minipage}[t]{0.5\linewidth}
\centering
\includegraphics[width=1\textwidth]{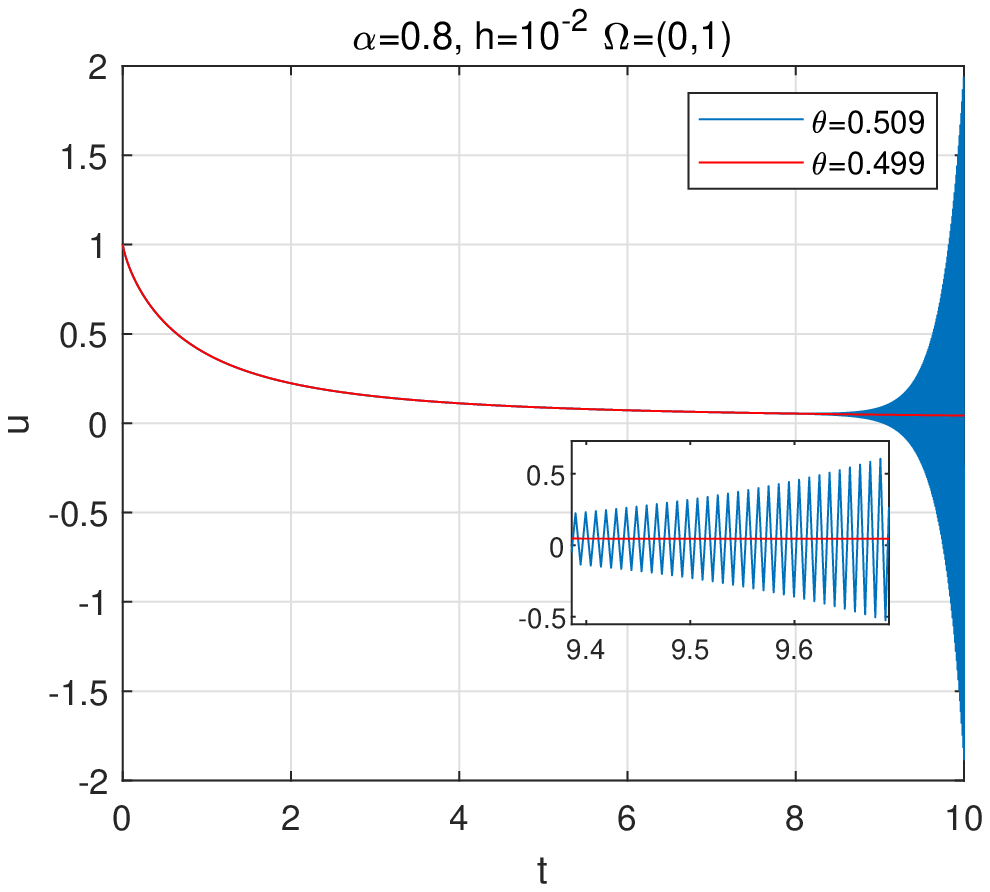}
\end{minipage}%
}%
\subfigure[]{
\begin{minipage}[t]{0.5\linewidth}
\centering
\includegraphics[width=1\textwidth]{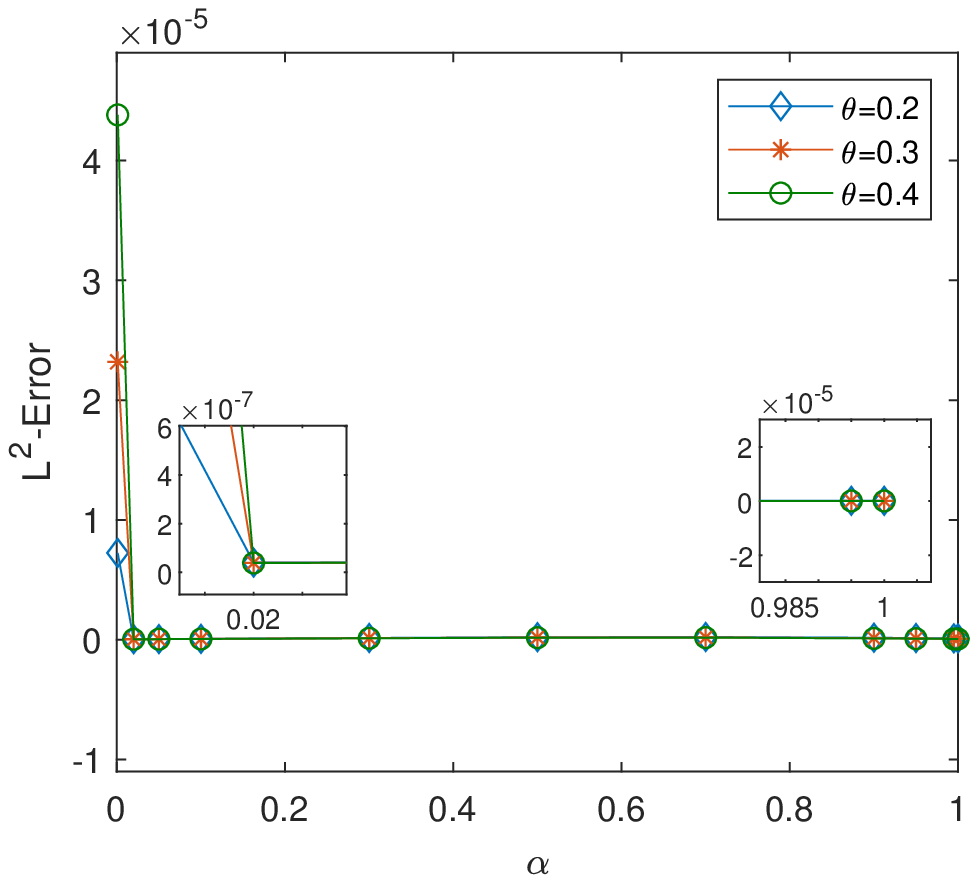}
\end{minipage}%
}%
\centering
\caption{(a) Check the stability of the scheme if $\theta>\frac{1}{2}$. (b) Check the robustness of the scheme if $\alpha \to 0^+$, or $\alpha \to 1^-$.}\label{C3}
\end{figure}
\\
\\
\textit{Example 4.}
\begin{table}[]
\centering
\caption{Temporal convergence rates of \textit{Example 4} at $t=0.5$.}\label{tab6}
{\small
{\renewcommand{\arraystretch}{1.2}
\begin{tabular}{ccrrrrrrrrr}
\toprule
$\alpha$             & $\theta$ & \multicolumn{1}{c}{$\tau=2^{-5}$} & \multicolumn{1}{c}{$\tau=2^{-6}$} & \multicolumn{1}{c}{$\tau=2^{-7}$} & \multicolumn{1}{c}{$\tau=2^{-8}$} & \multicolumn{1}{c}{$\tau=2^{-9}$} & \multicolumn{1}{l}{Rate} & \multicolumn{1}{l}{Rate} & \multicolumn{1}{l}{Rate} & \multicolumn{1}{l}{Rate} \\ \hline
\multirow{5}{*}{0.1} & 0.1      & 1.41E-05                          & 4.24E-06                          & 1.14E-06                          & 2.91E-07                          & 6.84E-08                          & 1.73                     & 1.89                     & 1.97                     & 2.09                     \\
                     & 0.2      & 2.12E-05                          & 5.89E-06                          & 1.54E-06                          & 3.88E-07                          & 9.26E-08                          & 1.85                     & 1.94                     & 1.99                     & 2.07                     \\
                     & 0.3      & 2.92E-05                          & 7.53E-06                          & 1.94E-06                          & 4.86E-07                          & 1.16E-07                          & 1.95                     & 1.96                     & 1.99                     & 2.06                     \\
                     & 0.4      & 9.25E-04                          & 1.04E-05                          & 2.33E-06                          & 5.83E-07                          & 1.41E-07                          & 6.47                     & 2.16                     & 2.00                     & 2.05                     \\
                     & 0.5      & 4.78E-01                          & 4.57E-01                          & 4.37E-01                          & 4.18E-01                          & 3.99E-01                          & 0.06                     & 0.06                     & 0.07                     & 0.07                     \\ \hline
\multirow{5}{*}{0.5} & 0.1      & 1.81E-04                          & 4.80E-05                          & 1.24E-05                          & 3.13E-06                          & 7.81E-07                          & 1.91                     & 1.96                     & 1.98                     & 2.00                     \\
                     & 0.2      & 2.17E-04                          & 5.66E-05                          & 1.44E-05                          & 3.64E-06                          & 9.07E-07                          & 1.94                     & 1.97                     & 1.99                     & 2.00                     \\
                     & 0.3      & 2.53E-04                          & 6.51E-05                          & 1.65E-05                          & 4.14E-06                          & 1.03E-06                          & 1.96                     & 1.98                     & 1.99                     & 2.00                     \\
                     & 0.4      & 4.54E-04                          & 7.36E-05                          & 1.85E-05                          & 4.65E-06                          & 1.16E-06                          & 2.62                     & 1.99                     & 1.99                     & 2.00                     \\
                     & 0.5      & 1.02E-01                          & 7.38E-02                          & 5.31E-02                          & 3.80E-02                          & 2.71E-02                          & 0.47                     & 0.48                     & 0.48                     & 0.49                     \\ \hline
\multirow{5}{*}{0.9} & 0.1      & 8.15E-04                          & 2.06E-04                          & 5.19E-05                          & 1.30E-05                          & 3.25E-06                          & 1.98                     & 1.99                     & 2.00                     & 2.00                     \\
                     & 0.2      & 8.60E-04                          & 2.17E-04                          & 5.46E-05                          & 1.37E-05                          & 3.41E-06                          & 1.99                     & 1.99                     & 2.00                     & 2.00                     \\
                     & 0.3      & 9.06E-04                          & 2.28E-04                          & 5.73E-05                          & 1.43E-05                          & 3.58E-06                          & 1.99                     & 1.99                     & 2.00                     & 2.00                     \\
                     & 0.4      & 9.72E-04                          & 2.39E-04                          & 6.00E-05                          & 1.50E-05                          & 3.75E-06                          & 2.02                     & 2.00                     & 2.00                     & 2.00                     \\
                     & 0.5      & 1.67E-02                          & 8.72E-03                          & 4.61E-03                          & 2.46E-03                          & 1.31E-03                          & 0.94                     & 0.92                     & 0.91                     & 0.90                     \\ \bottomrule
\end{tabular}}}
\end{table}
\par
This example is a further exploration for Crank-Nicolson scheme obtained by a different approximation formula for the derivative $D_t^\alpha u(t_{n-\frac{1}{2}})$.
To that end, we consider a general case, that approximating $D_t^\alpha u(t_{n-\theta})$ as follows
\begin{equation}\label{Nu.1}\begin{split}
D_t^\alpha u(t_{n-\theta})=\theta D_t^\alpha u(t_{n-1})+(1-\theta)D_t^\alpha u(t_{n})+O(\tau^2), \quad
\text{for sufficiently smooth $u$},
\end{split}\end{equation}
and then, applying the fractional BDF2 to approximate the term $D_t^\alpha u(t_{n})$ by
\begin{equation}\label{Nu.2}\begin{split}
D_\tau^\alpha u^n=\tau^{-\alpha}\sum_{k=0}^{n}\varpi_{n-k}u^k, \quad
\sum_{k=0}^{\infty}\varpi_k\xi^k=\varpi(\xi)=(3/2-2\xi+\xi^2/2)^\alpha.
\end{split}\end{equation}
We thus can obtain a full discrete scheme like (\ref{Se.14}) with the generating function $\omega(\xi)=(1-\theta+\theta\xi)\varpi(\xi)$.
\par
The empirical results in Table \ref{tab6} are obtained under the same conditions as Example 1, that $f=\big(6t^{3-\alpha}/\Gamma(4-\alpha)+t^3\big)\sin x$, $v=\sin x, x\in \Omega=(0,\pi)$ and $T=1$ with the exact solution $u(x,t)=(t^3+E_\alpha(-t^\alpha))\sin x$.
Numerically, we conclude the initial corrections designed initially for the SFTR, which disappear if $\theta=\frac{1}{2}$, are quite suitable for this example, provided $\theta \in (0,1/2)$.
If $\theta=\frac{1}{2}$, however, we can only get an $\alpha$-th order convergence rate, which is much lower than the second-order accuracy of the SFTR (see Table \ref{tab2} with $\theta=\frac{1}{2}$).
\\
\\
\textit{Example 5.}
\par
In this example, we compare the scheme (\ref{Se.14}) when $\theta=\frac{\alpha}{2}$, which is known as the fractional Crank-Nicolson scheme with the one studied in \cite{JinLiZh1} where a two-step initial correction technique was applied.
Further, we demonstrate the efficiency of the fast algorithm (\textbf{I}) and (\textbf{II}) developed for the SFTR, especially for long time simulations.
\par
Let $\Omega$ be an arbitrary convex polygonal domain in $\mathbb{R}^2$ with a quasi-uniform triangulation ($h=\frac{1}{20}$) as Fig.\ref{C4} (a).
Let $f=0$, $T=1$, and the initial data is given by
\begin{equation}\label{Nu.3}\begin{split}
v(x,y)=
\begin{cases}
  0.25, &  (x-1.7)^2+(y-1)^2<\frac{1}{32} \\
  0, &  (x-1.7)^2+(y-1)^2>\frac{1}{8} \\
  0.5-\sqrt{2(x-1.7)^2+2(y-1)^2}, & \mbox{otherwise},
\end{cases}
\end{split}\end{equation}
as depicted in Fig.\ref{C4} (b).
We take the numerical solution on a fine mesh $h=\frac{1}{200}$, $\tau=2^{-8}$ as the reference solution to calculate the convergence rates.
In Table \ref{tab7}, both schemes are performed to get the numerical solutions, from which one can conclude both of them are efficient.
In Fig.\ref{C5}, we show the evolution of the solution under $\alpha=0.5$ by the scheme (\ref{Se.14}).
Obviously, the numerical solution at final time with a decreasing height is smoother than the initial data.
\par
In Fig.\ref{C9}, by taking $h=\frac{1}{200}, T=1, \tau=2^{-8}$, we plot the relative errors of the fast algorithm solution $U_{hf}$ by the formula
\[
\max_{1\leq n \leq N}\frac{\|U_h^n-U^n_{hf}\|_{L^\infty}}{\|U_h^n\|_{L^\infty}}
\]
with varying $\theta$ for different $\alpha=0.01,0.3,0.6,0.99$.
One observes both fast algorithms require $\theta$ be bounded away from $0$ which is in line with the discussion in Sect.\ref{sec.fast}.
Moreover, both algorithms will deteriorate as $\theta$ tends to $\frac{1}{2}$, the reason of which deserves a further study.
For small $\alpha$ such as $\alpha=0.01$, Fast-II algorithm is superior to the first one by allowing a wider range of $\theta$.
In Fig.\ref{C10}, we illustrate the efficiency and accuracy of the fast algorithms for long time simulations.
By fixing $h=\frac{1}{100}$, $\tau=2^{-7}$ and taking $\alpha=0.3$, $\theta=0.1$, one can find the computing complexity of both fast algorithms (Fig.\ref{C10} (a)) is of $O(N\log N)$, which is much more efficient than the standard algorithm.
Both fast algorithms perform quite well for long time simulations as one can see in Fig.\ref{C10} (b) the relative errors are small enough and grow slowly with the advance of time.
\begin{figure}[htbp]
\centering
\subfigure[]{
\begin{minipage}[t]{0.5\linewidth}
\centering
\includegraphics[width=1\textwidth]{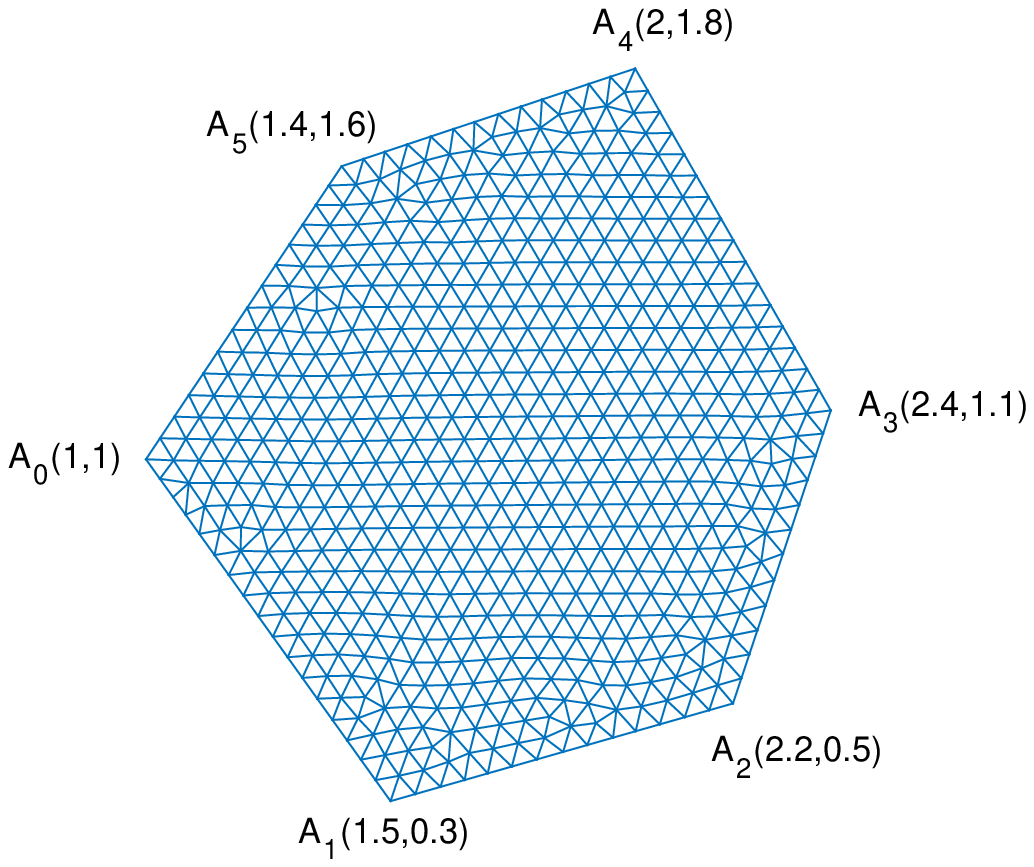}
\end{minipage}%
}%
\subfigure[]{
\begin{minipage}[t]{0.5\linewidth}
\centering
\includegraphics[width=1\textwidth]{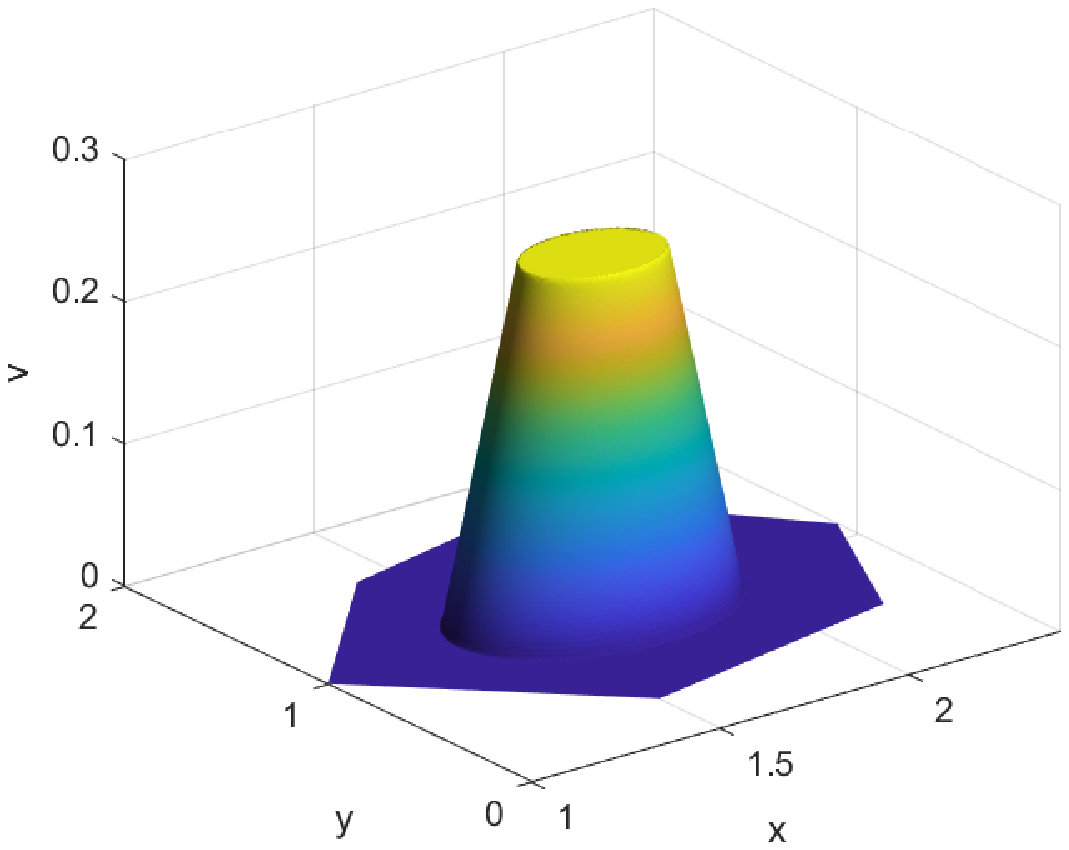}
\end{minipage}%
}%
\centering
\caption{(a) The region $\Omega$ with a quasi-uniform partition, $h=\frac{1}{20}$. (b) The nonsmooth initial value $v(x,y)$.}\label{C4}
\end{figure}

\begin{figure}[htbp]
\centering
\subfigure[]{
\begin{minipage}[t]{0.5\linewidth}
\centering
\includegraphics[width=1\textwidth]{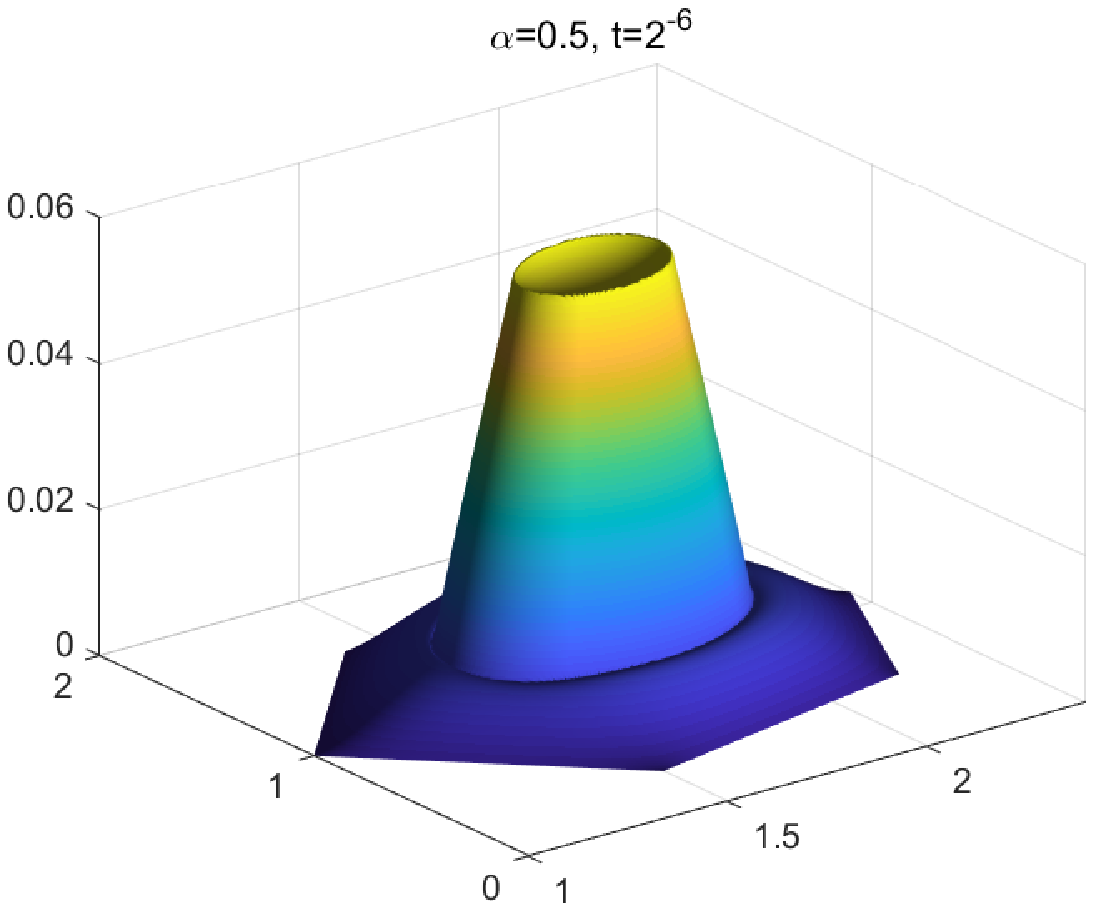}
\end{minipage}%
}%
\subfigure[]{
\begin{minipage}[t]{0.5\linewidth}
\centering
\includegraphics[width=1\textwidth]{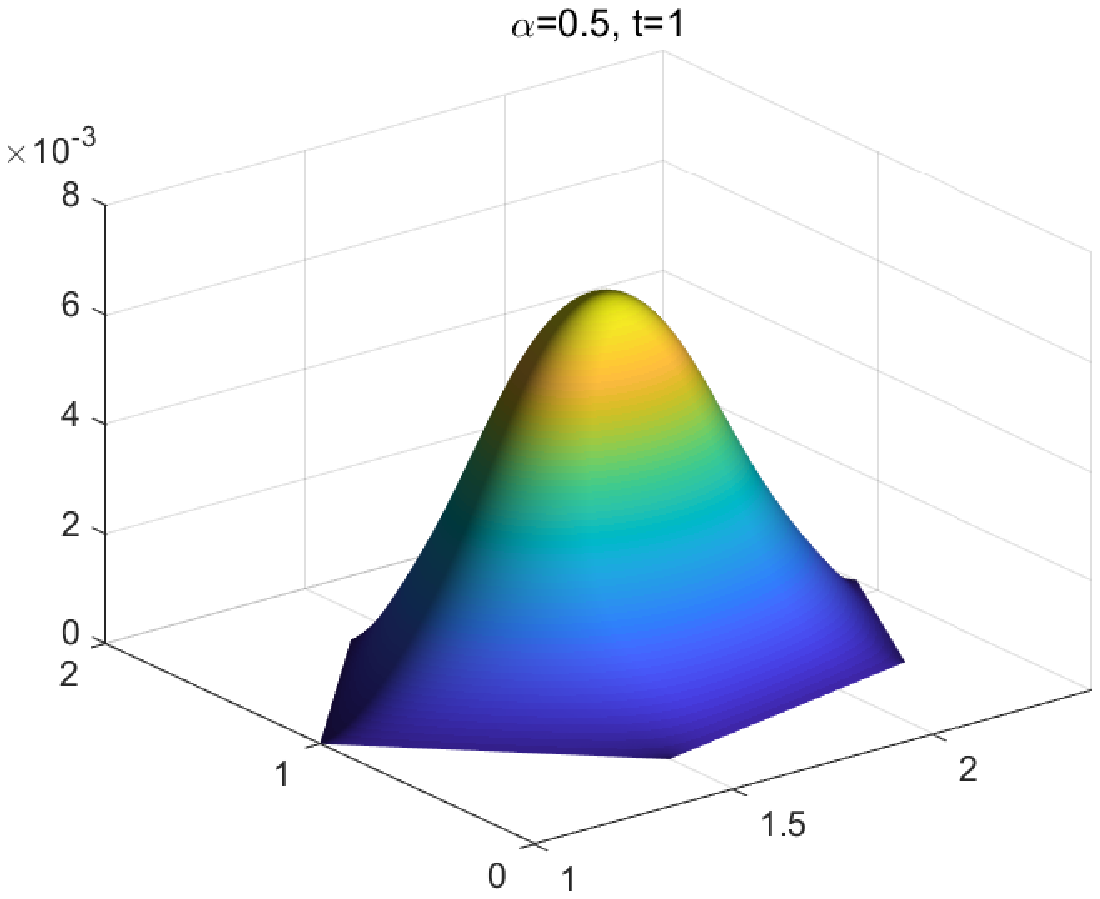}
\end{minipage}%
}%
\centering
\caption{(a) Numerical solution at time $t=2^{-6}$, $\alpha=0.5$. (b) Numerical solution at time $t=1$, $\alpha=0.5$.}\label{C5}
\end{figure}
\begin{figure}[htbp]
\centering
\subfigure[]{
\begin{minipage}[t]{0.5\linewidth}
\centering
\includegraphics[width=1\textwidth]{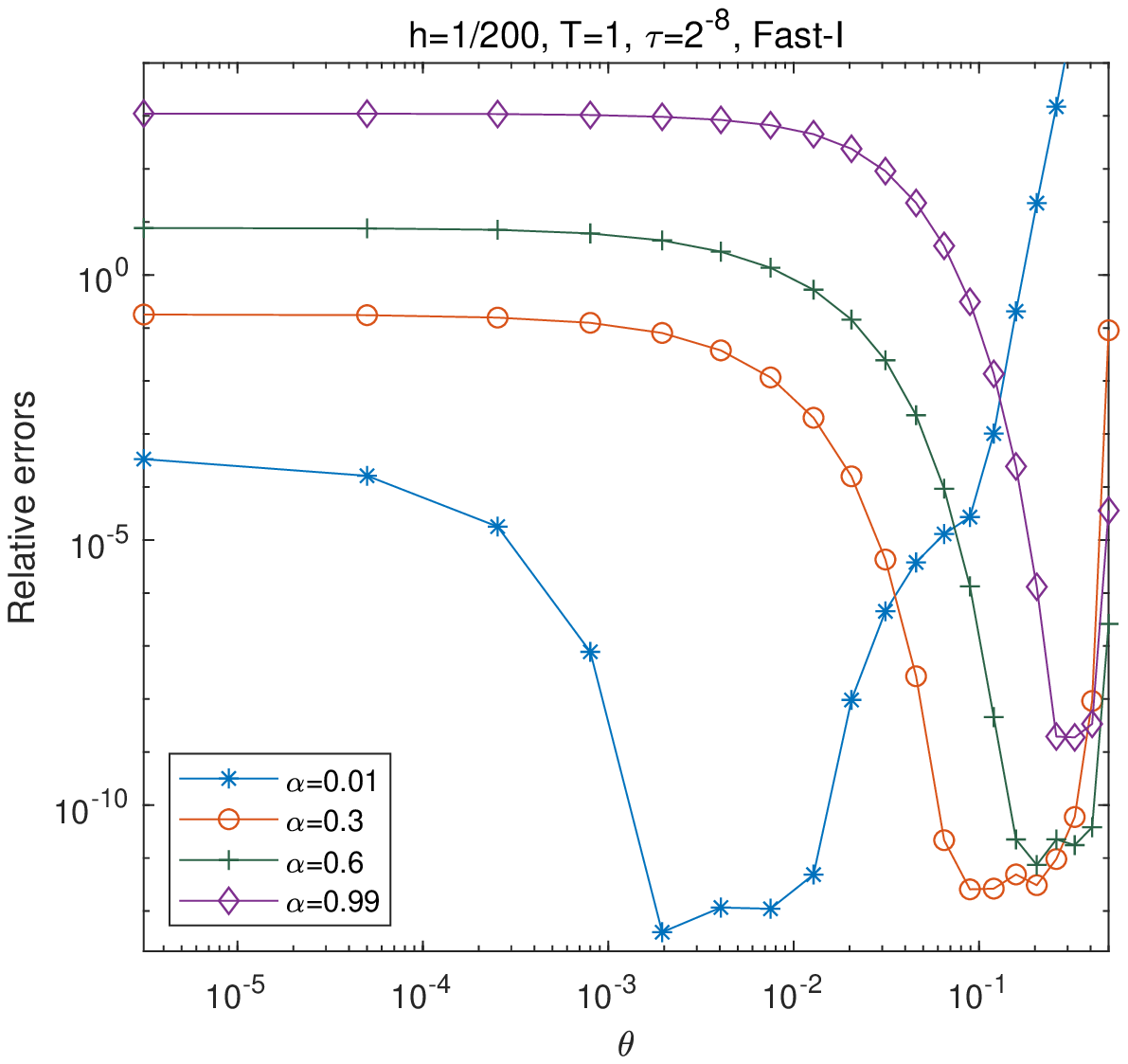}
\end{minipage}%
}%
\subfigure[]{
\begin{minipage}[t]{0.5\linewidth}
\centering
\includegraphics[width=1\textwidth]{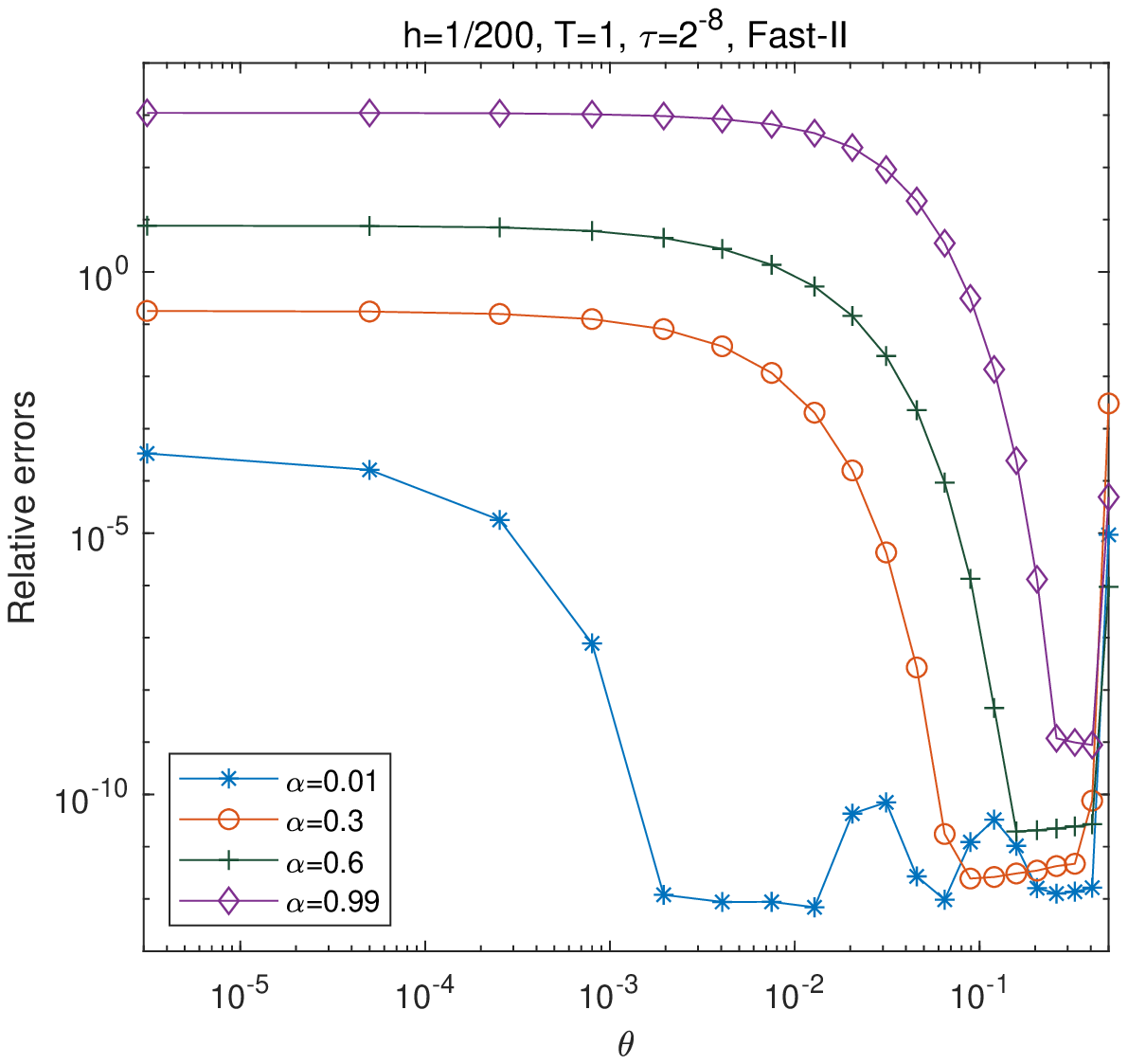}
\end{minipage}%
}%
\centering
\caption{Relative errors for fast algorithms (\textbf{I})(a) and \textbf{II}(b) with different $\theta$. $h=\frac{1}{200}, T=1, \tau=2^{-8}$.}\label{C9}
\end{figure}
\begin{figure}[htbp]
\centering
\subfigure[]{
\begin{minipage}[t]{0.5\linewidth}
\centering
\includegraphics[width=1\textwidth]{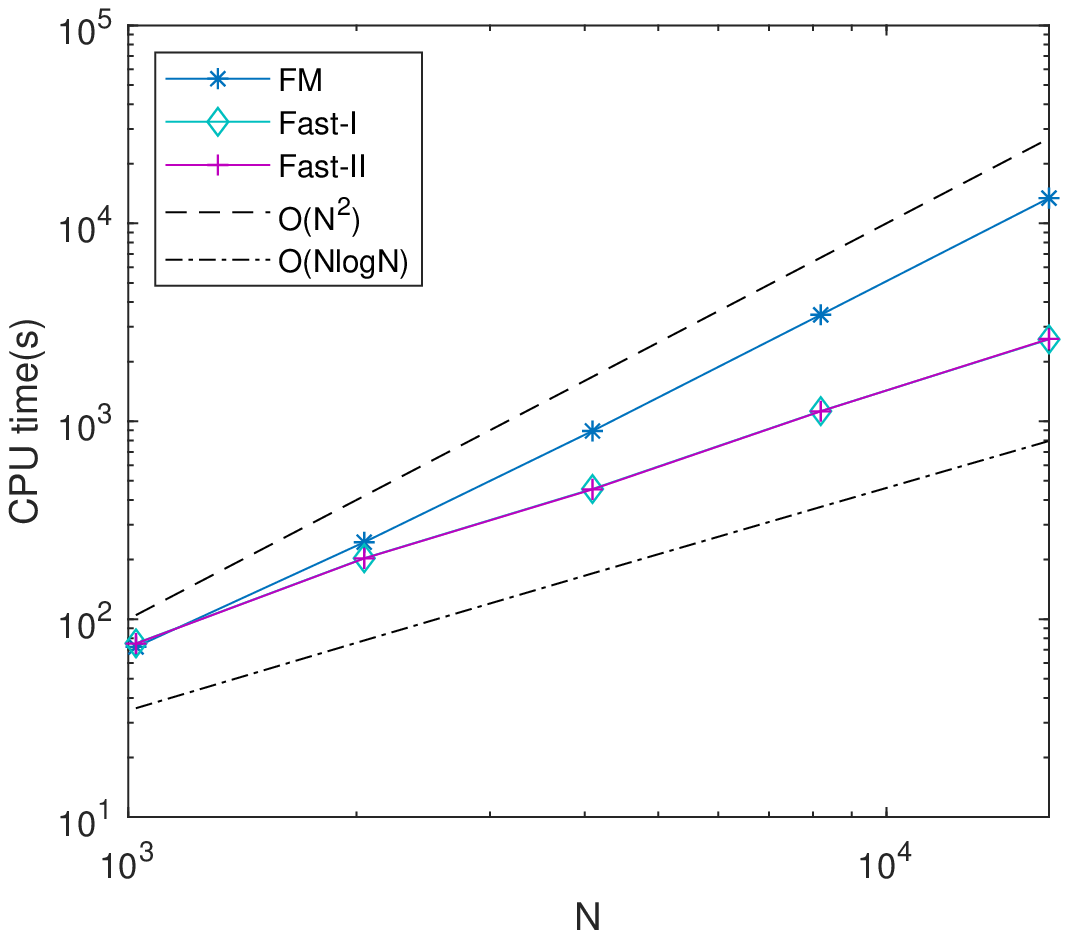}
\end{minipage}%
}%
\subfigure[]{
\begin{minipage}[t]{0.5\linewidth}
\centering
\includegraphics[width=1\textwidth]{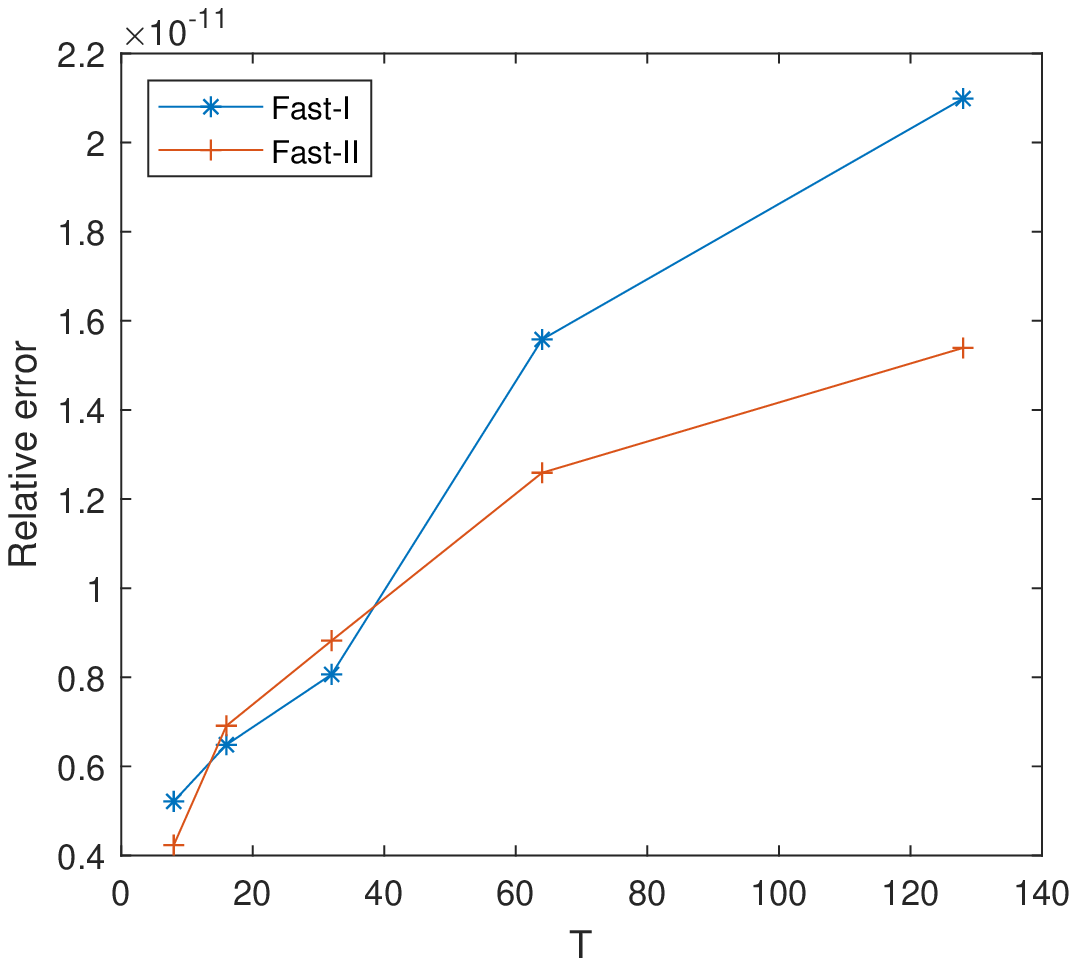}
\end{minipage}%
}%
\centering
\caption{Setting $h=\frac{1}{100}$, $\tau=2^{-7}$, $\alpha=0.3$ and $\theta=0.1$. (a) CPU time consumed for fast algorithms I and II. (b) Relative errors obtained for fast algorithm I and II.}\label{C10}
\end{figure}
\begin{table}[]
\centering
\caption{Temporal convergence rates of \textit{Example 5} at $t=0.5$.}\label{tab7}
{\small
{\renewcommand{\arraystretch}{1.2}
\begin{tabular}{crcclcc}
\toprule
\multicolumn{1}{l}{}         & \multicolumn{1}{l}{}       & \multicolumn{2}{l}{Scheme (\ref{Se.14})}                           &  & \multicolumn{2}{l}{Scheme in \cite{JinLiZh1}}                           \\ \cline{3-4} \cline{6-7}
\multicolumn{1}{l}{$\alpha$} & \multicolumn{1}{l}{$\tau$} & \multicolumn{1}{l}{Error} & \multicolumn{1}{l}{Rate} &  & \multicolumn{1}{l}{Error} & \multicolumn{1}{l}{Rate} \\ \hline
\multirow{4}{*}{0.1}         & 1/8                        & 1.53E-05                  & \multicolumn{1}{l}{}     &  & 2.24E-05                  & \multicolumn{1}{l}{}     \\
                             & 1/16                       & 4.21E-06                  & 1.86                     &  & 4.75E-06                  & 2.24                     \\
                             & 1/32                       & 9.68E-07                  & 2.12                     &  & 1.24E-06                  & 1.93                     \\
                             & 1/64                       & 2.23E-07                  & 2.12                     &  & 4.47E-07                  & 1.48                     \\ \hline
\multirow{4}{*}{0.5}         & 1/8                        & 1.13E-03                  & \multicolumn{1}{l}{}     &  & 3.90E-04                  & \multicolumn{1}{l}{}     \\
                             & 1/16                       & 2.17E-05                  & 5.70                     &  & 3.17E-05                  & 3.62                     \\
                             & 1/32                       & 5.09E-06                  & 2.09                     &  & 7.15E-06                  & 2.15                     \\
                             & 1/64                       & 1.19E-06                  & 2.10                     &  & 1.63E-06                  & 2.13                     \\ \bottomrule
\end{tabular}}}
\end{table}
\section{Conclusion and discussion}\label{sec.con}
In this study, we analyze a class of time-stepping methods with initial corrections for subdiffusion problems.
Stability analysis is provided by using the maximal $\ell^p$-regularity and a discrete Gr\"{o}nwall type inequality.
Sharp error estimates which depend on the regularity of the initial data and source term are presented based on the Laplace transform and estimates of kernel functions.
Moreover, we consider the dependence on the fractional order $\alpha$ for constants $c$ involved in error estimates, showing that our estimates are $\alpha$-robust.
To speed up the calculation in simulations, two different fast algorithms are developed for the SFTR with some discussions on the choice of the parameter $\theta$.
Several numerical experiments are performed to guarantee the robustness of our scheme and confirm the efficiency of the initial correction technique and fast algorithm.
\par
As is shown in the numerical tests section, if the shifted parameter $\theta$ is chosen as $\frac{1}{2}$, which is theoretically excluded from this study, the empirical results are rather amazing as the optimal convergence rate is recovered for smooth initial data (still results in a weak regular solution) although the initial corrections vanish in this case.
We emphasize here such phenomenon is closely related to the approach one employs to discretize the fractional derivative, or, from the aspect of generating functions $\omega(\xi)$, to singularity points of $\omega(\xi)$.
Developing analysis for the special case $\theta=\frac{1}{2}$ under different regularity of initial data, and designing initial corrections if convergence rates are not optimal are challengable, and will be our future work.
\section*{Acknowledgements}
The work of the second author was supported in part by grants NSFC 12061053, 11661058 and the NSF of Inner Mongolia 2020MS01003.
The work of the third author was supported in part by the grant NSFC 11761053. The work of the fourth author was supported in part by grants NSFC 11871092 and NSAF U1930402.

\end{document}